\documentclass[10pt]{article}

\usepackage{setspace}
\renewcommand\labelenumi{(\roman{enumi})}
\renewcommand\theenumi\labelenumi

\usepackage[colorlinks=true,linkcolor=blue,citecolor=red]{hyperref}

\usepackage{authblk}

\usepackage[T1]{fontenc}

\usepackage{thmtools}
\usepackage{amssymb}
\usepackage{xcolor}
\usepackage{enumitem}
\usepackage{mathtools}
\usepackage{geometry}
\usepackage{amsfonts}
\usepackage{amsmath}
\usepackage{amsthm}
\usepackage{mathrsfs}
\usepackage{esint,subfig}
\usepackage{latexsym}
\usepackage[numbers]{natbib}
\usepackage{bm,bbm}
\usepackage{caption}
\usepackage{float}
\usepackage{multicol}
\usepackage{graphicx}
\usepackage{appendix}
\usepackage{lmodern}
\usepackage{xr}

\allowdisplaybreaks

\numberwithin{equation}{section}

\newtheorem{theorem}{Theorem}[section]

\newtheorem{condition}[theorem]{Condition}

\newtheorem{definition}[theorem]{Definition}
\newtheorem{example}[theorem]{Example}

\newtheorem{lemma}[theorem]{Lemma}

\newtheorem{proposition}[theorem]{Proposition}
\newtheorem{remark}[theorem]{Remark}

\definecolor{bittersweet}{rgb}{1.0, 0.44, 0.37}
\definecolor{electricultramarine}{rgb}{0.25, 0.0, 1.0}
\definecolor{rred}{RGB}{152,0,0}
\definecolor{carminepink}{rgb}{0.92, 0.3, 0.26}

\renewcommand\labelenumi{(\roman{enumi})}
\renewcommand\theenumi\labelenumi

\newcommand{\Ex}{\mathbb{E}}				  
\newcommand{\Prob}{\mathbb{P}}

\newcommand{\Var}{{\mathbb{V}}}

\newcommand{\s}{{\mathbb{S}^2}}

\newcommand{\ls}{L^2(\s)}

\newcommand{\LSr}{L^2(\s; \mathbb{R})}
\newcommand{\LSc}{L^2(\s; \mathbb{C})}
\newcommand{\LSh}{{\mathbb{H}}}
\newcommand{\op}{{\operatorname{op}}}

\newcommand{\ylm}{Y_{\ell, m}}
\newcommand{\ylmc}{\overline{Y_{\ell, m}}}

\renewcommand{\{}{\left\lbrace}
\renewcommand{\}}{\right\rbrace}

\newcommand{\SPHAR}{\operatorname{SPHAR}}   
\newcommand{\ARp}{\operatorname{AR}(p)}

\let\theta\vartheta
\let\itheta\temp

\def\1{\mathbbm{1}}

\newlength\mybibindent


\newcommand{\saveparinfos}{%
\edef\myindent{\the\parindent}%
\edef\myparskip{\the\parskip}}

\begin{document}

\author[1]{Alessia Caponera\footnote{e-mail: caponera@mat.uniroma2.it \\ The research was supported by MIUR Excellence Department Project awarded to the Department of Mathematics, University of Rome Tor Vergata, CUP E83C18000100006. The author is also grateful to Domenico Marinucci for many insightful discussions and suggestions.}}
\affil[1]{Dipartimento di Matematica, Tor Vergata Universit\`a di Roma}

%\cortext[cor1]{}
%\fntext[fn1]{E-mail address:alessia.caponera@uniroma1.it.}
%\fntext[fn2]{E-mail address:}
%\fntext[fn3]{E-mail address:vidotto@mat.uniroma2.it.}

\title{SPHARMA approximations for stationary functional time series on the sphere}
%\author[1]{A. Caponera\thanks{alessia.caponera@uniroma1.it}}
%\author[2]{C. Durastanti\thanks{claudio.durastanti@uniroma1.it}}
%\author[3]{A. Vidotto\thanks{vidotto@mat.uniroma2.it}}
%\affil[1]{{\small Dipartimento di Scienze Statistiche, Universit\`a degli Studi di Roma ``La Sapienza''}}
%\affil[2]{{\small Dipartimento SBAI, Universit\`a degli Studi di Roma ``La Sapienza''}}
%\affil[3]{{\small Dipartimento di Matematica, Universit\`a degli Studi di Roma ``Tor Vergata''}}

%\renewcommand\Authands{ and }

\date{\today}

\maketitle
\begin{abstract}
In this paper, we focus on isotropic and stationary sphere-cross-time random fields. We first introduce the class of spherical functional autoregressive-moving average processes (SPHARMA), which extend in a natural way the spherical functional autoregressions (SPHAR) recently studied in \cite{cm19,cdv19}; more importantly, we then show that SPHAR and SPHARMA processes of sufficiently large order can be exploited to approximate \textit{every} isotropic and stationary sphere-cross-time random field, thus generalizing to this infinite-dimensional framework some classical results on real-valued stationary processes. Further characterizations in terms of functional spectral representation theorems and Wold-like decompositions are also established.

\textit{Keywords:} time-varying spherical random fields, functional
time series, double spectral representation, spherical harmonics, spherical
functional ARMA\newline
\textit{MSC 2010 subject classications}: primary 62M15; secondary 62M10, 60G15, 60F05, 62M40, 60G60.
\end{abstract}

\affil[1]{Dipartimento di Matematica, Tor Vergata Universit\`a di Roma}

%\cortext[cor1]{}
%\fntext[fn1]{E-mail address:alessia.caponera@uniroma1.it.}
%\fntext[fn2]{E-mail address:}
%\fntext[fn3]{E-mail address:vidotto@mat.uniroma2.it.}

%\author[1]{A. Caponera\thanks{alessia.caponera@uniroma1.it}}
%\author[2]{C. Durastanti\thanks{claudio.durastanti@uniroma1.it}}
%\author[3]{A. Vidotto\thanks{vidotto@mat.uniroma2.it}}
%\affil[1]{{\small Dipartimento di Scienze Statistiche, Universit\`a degli Studi di Roma ``La Sapienza''}}
%\affil[2]{{\small Dipartimento SBAI, Universit\`a degli Studi di Roma ``La Sapienza''}}
%\affil[3]{{\small Dipartimento di Matematica, Universit\`a degli Studi di Roma ``Tor Vergata''}}

%\renewcommand\Authands{ and }

\section{Introduction}

\label{sec:intro}

Over the last few years, the analysis of sphere-cross-time isotropic and stationary random fields has drawn a considerable amount of attention, due to strong motivations arising in Climate and Atmospheric Sciences, Geophysics, Astrophysics and Cosmology, and many other areas of research, see for instance %\cite{BKMP,Berg,Cheng2019,Clarke,Fan,Fan2,Gneiting,Leonenko,Porcu18,Porcu} 
\cite{castruccio2016, castruccio2013, Ch:05, Leonenko, Porcu18, porcu, white2019} and the references therein. A lot of efforts has been spent, in particular, on the characterization of covariance functions and their use for parametric inference, see \cite{bergporcu, gneiting, JJ:15, jun, Porcu18, porcu, white2019} and in particular \cite{PFN:20} for a comprehensive review. %\cite{Clarke, Gneiting, Jun, Porcu}

More recently, spherical functional autoregressive models (SPHAR) have been introduced and studied in \cite{cm19,cdv19}, where they were shown to provide a flexible tool for the analysis of time-dependent spherical data. In particular, in \cite{cm19} estimators based on a functional $L^2$-minimization criterion have been investigated, whereas in \cite{cdv19} LASSO-type penalized extensions were considered; indeed, various forms of concentration properties, laws of large numbers, quantitative and functional central limit theorems were established under broad general assumptions.

The purpose of the present work is to show that the spherical autoregression framework, and its natural generalization to spherical autoregressive moving averages (SPHARMA), provides a very general approximation for \textit{every} isotropic and stationary sphere-cross-time random field. In particular, after introducing rigorously the SPHARMA class we show in our main results below (Theorems \ref{theo::approx-ma}, \ref{theo::approx-ar} and \ref{theo::approx-l2}) that any isotropic and stationary spherical random field can be approximated, in terms of their harmonic transforms and in the $L^2(\Omega)$ sense, by a SPHARMA or SPHAR process of sufficiently large order. %From a different point of view, this result can be considered as a generalization for functional time series with values in $L^2(\mathbb{S}^2)$ of well-known approximations for real-valued stationary time series, see \cite[Chapters 4 and 5]{brockwelldavis}. 

Of course, sphere-cross-time data can also be viewed as functional time series. In this sense, we can exploit the rich machinery recently developed in this area; the reader is referred to \cite{bosq, BB:07, HK:12, Hsing, RS:05} for an overview. In particular, 
as for the finite-dimensional case, the analysis of second-order properties is a crucial tool in the characterization of stationary functional processes. More precisely,
some recent works have focused on a functional frequency-domain approach, encoding the complete second-order structure of stationary sequences, via
the spectral density operator, i.e. the Fourier transform of the collection of autocovariance operators,
\begin{align*}
& \mathscr{F}_\lambda := \frac{1}{2\pi} \sum_{t \in \mathbb{Z} }
e^{-i\lambda t} \mathscr{R}_t, \qquad \lambda \in [-\pi,\pi].
\end{align*}
Such operators were first investigated in \cite{Panaretos2}; then, \cite{Panaretos} derived a Cram{\'e}r--Karhunen--Lo{\`e}ve representation for \textit{short-memory} processes;
we specialize this result to the spherical case to make this paper self-contained and complete. Loosely speaking, this representation first decomposes the series into an integral
of uncorrelated frequency components (Cram{\'e}r representation), each of which is in turn expanded in a Karhunen--Lo{\`e}ve series, by means of
eigenfunctions of $\mathscr{F}_\lambda$; this way at the same time it provides a full description of the second-order dynamics and it gives some insights into an optimal finite-dimensional
representation. Very recently, \cite{vdE:20} establish the functional versions of Herglotz's Theorem and the Cram{\'e}r representation under more general assumptions, which cover also stationary Hilbert-valued time series with discontinuities in the spectral measure and \textit{long-memory} processes. Overall, such results laid the foundations for a variety of frequency
domain-based inference procedures in functional time series, e.g. \cite{Aue,Hormann}.

Following these lines, we start with the investigation of the spectral characteristics of time-varying isotropic spherical random fields; under these circumstances it is possible to derive a double spectral decomposition based on
spherical harmonics. Then, we focus on harmonic properties of functional autoregressive-moving average processes (see \cite{bosq, cm19}), defined as random elements of $L^2(\mathbb{S}^2)
$ (as mentioned above, recently \cite{cm19, cdv19} have addressed the estimation of the functional linear autoregressive 
operators). Our approach is related to the extensive literature on linear stationary functional processes; however most of this literature is built on the estimation of autocovariance operators in the time domain, see among others \cite{Bo:02,DS:05,Ma:02}.

After building this background, we arrive at our main results; as anticipated, we prove that, under very broad assumptions, any isotropic and stationary spherical process can be approximated arbitrary well by a SPHARMA model of sufficiently large degree. Our results can be viewed as the infinite-dimensional counterpart of the rational approximations of spectral densities and the Wold decomposition in the analysis of standard, real-valued stationary processes, see for instance
 \cite[Chapters 4 and 5]{brockwelldavis}.

\paragraph{Plan of the paper.} 
The paper is organized as follows. In Section \ref{sec::spectral-char}, we describe the spectral properties of isotropic stationary sphere-cross-time random fields, while in Section \ref{sec::spectral-repr} we provide a further characterization in terms functional spectral representation theorems. In Section \ref{sec::spharma}, we introduce rigorously the SPHARMA class. Section \ref{sec::main-results} contains the main results of this work,
that is, SPHARMA approximations of stationary and isotropic spherical random field, in terms of their harmonic transforms and in the $L^2(\Omega)$ sense, together with a Wold-like decomposition. Lastly, Section \ref{sec::proofs} collects the proofs.

\paragraph{Notation.}
We will denote with $\mathfrak{B}(\s)$ the Borel $\sigma$-field on the unit sphere and with $\LSc:=L^2(\s, dx; \mathbb{C})$ the Hilbert space of square-integrable complex-valued functions on $\s$ endowed with the usual inner product $\langle f, g \rangle_{\LSc} = \int_\s f(x)\overline{g(x)}d x$. $\|\cdot\|_{\LSc}$ will be the norm induced by $\langle \cdot , \cdot \rangle_{\LSc}$; to simplify the notation, sometimes we will replace the subscript $\LSc$ with $2$. 
Moreover, the restriction of $\LSc$ to real-valued functions will be denoted by $\LSr$. We will also use the same notation for $\textnormal{L}^2(\s \times \s; \mathbb{C})$ and $\textnormal{L}^2(\s \times \s; \mathbb{R})$.
Let $\LSh$ be the Hilbert space of $\LSc$-valued random elements with finite second moment, that is, $f \in \LSh$ is such that $\Ex \| f \|^2_{\LSh} < \infty$. The associated inner product is defined as $\langle f, g \rangle_{\LSh}= \Ex \langle f, g \rangle_{\LSc},  \text{ for } f,g \in \mathbb{H}$.
For $u,v \in \LSc$, the tensor product $u \otimes v$ is defined to be the mapping that takes any element $f \in \LSc$ to $u \langle f, v\rangle \in \LSc.$
$\|\mathscr{T}\|_{\operatorname{TR}}$ is the trace (or nuclear) norm of the operator $\mathscr{T}$, see \cite{Hsing}. For a real- or complex-valued function $f$ defined on a set $D$, we define $\|f\|_\infty := \sup_{x \in D} |f(x)|$. $\delta_a^b$
is the Kronecker delta function.

\section{Spectral characteristics}\label{sec::spectral-char}

Consider the collection of random variables $\{T(x,t), \ (x,t) \in \s \times \mathbb{Z} \}$ defined on the probability space $(\Omega, \mathfrak{F}, \Prob)$. For every fixed $t \in \mathbb{Z}$, $\{T(x, t), \ x \in \s \}$ is a spherical random field as defined in \cite[Chapter 5]{MP:11}; recall that we are implicitly assuming measurability with respect to the product $\sigma$-field $\mathfrak{B}(\s) \times \mathfrak{F}$. We name $\{T(x,t), \ (x,t) \in \s \times \mathbb{Z} \}$ \emph{space-time spherical random field}. For simplicity, we will assume that $\Ex[T(x,t)] = 0,$ for all $(x,t) \in \s \times \mathbb{Z}$. 

%As usual in the context of functional data analysis, we model the mapping $T_t: \omega \mapsto T(\cdot, t, \omega)$ as a random element of the separable Hilbert space of functions $\LSr$.

For the rest of the paper, we are going to consider space-time spherical random fields which are jointly isotropic (in the spatial component) and stationary (in the temporal component). To this purpose, we give the following definition (see also \cite{clarke}).

\begin{definition}\label{def::isotropy-stationarity}
We say that the zero-mean space-time spherical random field $\{T(x,t), \ (x,t) \in \s \times \mathbb{Z} \}$ is isotropic stationary if $\Ex |T(x,t) |^2 < \infty$, for all $(x,t) \in \s \times \mathbb{Z}$, and
\begin{equation*}
\Ex[T(x,t)T(y,s)] =\Ex[T(gx, t+h) T(gy, s+h)],
\end{equation*}
for all $x,y \in \s, \ g \in SO(3), \ t,s,h \in \mathbb{Z}$.
\end{definition}
%Thus, in this case, as for standard time series, we can simplify the notation as follows: 
%\begin{equation}\label{stationary_autocov}
%\mathscr{R}_{t} \equiv \mathscr{R}_{t,0}=\Ex \left [ T_t \otimes T_0 \right ], \qquad t \in \mathbb{Z}.
%\end{equation}

Thus, in this case we can define the autocovariance kernel at lag $t$
\begin{equation}\label{covkern}
r_t(x,y) = \Ex[T(x, t)T(y, 0)], \qquad x,y \in \s, \ t \in \mathbb{Z},
\end{equation}
which is also Hilbert-Schmidt, i.e. $r_t(\cdot,\cdot ) \in \textnormal{L}^2(\s \times \s, \mathbb{R})$, and the corresponding operator $\mathscr{R}_t: \LSc \to \LSc$ induced by right integration, the autocovariance operator at lag $t$, 
\begin{equation*}
(\mathscr{R}_t h) (\cdot) = \int_\s r_t(\cdot,y) h(y) d y, \qquad h \in \LSc.
\end{equation*}

%\begin{remark}
%Note that in \cite{Panaretos} the autocovariance operators are defined on the restriction of $L^2([0,1])$ to real-valued functions. Here, even if we work with random elements of $\LSr$, their spectral representation is in terms of the standard complex basis of spherical harmonics $\{\ylm\}$, hence we shall define the $\mathscr{R}_{t}$'s on $\LSc$.
%\end{remark}

%\begin{definition}
%The space-time spherical random field $ \{ T(x,t),\ x \in \mathbb{S}^2\times \mathbb{Z}\}$ is said to be isotropic stationary if we have that $(T(x_1,t_1), \ldots, T(x_k,t_k))'$ and $(T(g x_1,t_{1} + h ), \ldots, T(g x_k,t_{k}+h))'$ have the same joint distribution, for every $k \in \mathbb{N}$, every $x_1, \ldots, x_k \in \mathbb{S}^2 $, every $g \in SO(3)$ and every $t_1, \ldots, t_k, h \in \mathbb{Z}$.
%\end{definition}

In functional data analysis, it is usual to model random processes as random elements of some separable Hilbert space. Under the joint isotropy-stationarity assumption, the sequences of spherical random fields we are considering can be seen as a sequence of random elements of $L^2(\mathbb{S}^2)$. 
More formally, for any fixed $t \in \mathbb{Z}$, there exists a random element $T_t$ of $L^2(\mathbb{S}^2)$ such that $T(\cdot, t) = T_t \ \mathbb{P}-$a.s., indeed 
$$
\mathbb{E} \int_{\mathbb{S}^2} |T(x,t)|^2dx = 4\pi \mathbb{E}[T(x_0,t)^2] < \infty,
$$
for any $x_0 \in \mathbb{S}^2$.
This implies that there exists a $\mathfrak{F}$-measurable set $\Omega'$ of $\mathbb{P}$-probability 1 such that, for every $\omega \in \Omega'$,  $T(\cdot,t,\omega)$ is an element of $L^2(\mathbb{S}^2)$. 

\begin{remark}
Recall that if $\{ T(x), \ x \in \s \}$ is jointly measurable and $T(\cdot, \omega) \in \LSr$ for each $\omega$, then the mapping $\omega \mapsto T(\cdot, \omega)$ is a random element of $\LSr$ (see \cite[Theorem 7.4.1]{Hsing}).
\end{remark}

Thus, $\{T_t, \ t \in \mathbb{Z}\}$ is a stationary sequence of random elements in $\LSr$, with mean zero and $\Ex\| T_0 \|^2 <\infty$; see for instance \cite[Definition 2.4]{bosq} and $\mathscr{R}_t$ coincides with the autocovariance operator defined as the Bochner integral
\begin{equation*}
\Ex \left [ T_t \otimes T_0 \right ] := \int_{\Omega} T_t \otimes T_0 \, d  \Prob.
\end{equation*}

In this setup, it is possible to show that the following representation holds
\begin{equation*}
T(x,t) =\sum_{\ell=0}^{\infty} \sum_{m=-\ell}^{\ell}a_{\ell,m}(t)Y_{\ell,m}(x),
\end{equation*}
in the $L^2(\Omega)$ sense for every $(x,t) \in \mathbb{S}^2\times \mathbb{Z}$ and in the $L^2(\s \times \Omega)$ sense for every $t \in \mathbb{Z}$; the sequence $\{Y_{\ell,m}, \ell \ge, m=-\ell,\dots,\ell\}$ is a standard orthonormal basis for $\LSc$ of spherical harmonics, whereas, 
%that is, for every $(x,t) \in \mathbb{S}^2\times \mathbb{Z}$,
%$$
%\lim_{L\rightarrow\infty} \mathbb{E} \left[ \left(T(x,t) - \sum_{\ell=0}^{L} \sum_{m=-\ell}^{\ell}a_{\ell,m}(t)Y_{\ell,m}(x)\right)^2 \right] = 0.
%$$
for fixed $t \in \mathbb{Z}$, $\{a_{\ell,m}(t), \ell \ge, m=-\ell,\dots,\ell \}$ is a triangular array of zero-mean uncorrelated random coefficients defined as $$a_{\ell,m}(t)= \int_{\mathbb{S}^2} T(x,t) \overline{Y_{\ell,m}(x)} d x.$$
This result can be understood as a direct application of the spectral theorem for isotropic random fields on the sphere, see \cite[Chapter 5 and in particular Theorem 5.13]{MP:11}. In this sense, it does not give insights on the temporal dynamics of the process and, hence, on its complete second-order structure. %his argument will be addressed in Section \ref{sec::spectral-repr}.

%In this setup, most of the conditions given in \cite{Panaretos} are satisfied.
%$T_t$ is a stationary sequence of random elements in $\LSr$, with $\Ex \| T_t\|^2 < \infty$, and each random field $\{T(x, t), \ x \in \s\}$,  $t \in \mathbb{Z},$ is mean-square continuous, that is,
%\begin{equation}\label{mean_square_cont}
%\lim_{x \to x_0} \Ex[ T(x,t ) - T(x_0, t) ]^2 = 0, \qquad \forall x_0 \in \s.
%\end{equation}
%
%The autocovariance kernel at lag $t$
%\begin{equation}
%r_t(x,y) = \Ex[T(x, t)T(y, 0)], \qquad x,y \in \s, \ t \in \mathbb{Z},
%\end{equation}
%is Hilbert-Schmidt, i.e. $r_t(\cdot,\cdot ) \in \textnormal{L}^2(\s \times \s, \mathbb{R})$, and it is continuous on $\s \times \s$. 
%
%\begin{remark}
%Continuity of all kernels $r_t(\cdot, \cdot), \ t \in \mathbb{Z},$ (see Condition 1.1 (ii) in \cite{Panaretos}) follows from mean-square continuity of $\{T(x, t), \ x \in \s\}, \ t \in \mathbb{Z},$ (which is in turn consequence of isotropy, see \cite{MP:11}). 
%Moreover, notice that if $r_0(\cdot, \cdot)$ is continuous on $\s \times \s$, then each random field $\{T(x, t), \ x \in \s\}$,  $t \in \mathbb{Z},$ is mean-square continuous.
%\end{remark}

Following \cite{Panaretos2}, we shall use the conditions below to define the spectral density kernels and the spectral density operators and to prove part of our main results in Section \ref{sec::main-results}. Under these conditions, we are also able to give first a Functional Cram{\'e}r Representation which involves a $\LSc$-valued orthogonal increment process, and then to obtain a double spectral representation with respect to both space and time, see Section \ref{sec::spectral-repr} below. We stress that in this section and in Section \ref{sec::spectral-repr}, as in \cite{Panaretos}, it is not assumed any other prior structural properties for the stationary sequence (e.g., linearity or Gaussianity).

\begin{condition}\label{cond::summability} For an isotropic stationary space-time spherical random field $\{T(x,t), \ (x,t) \in \mathbb{S}^2 \times \mathbb{Z}\}$ with continuous covariance kernels \eqref{covkern} on $\s \times \s$, consider one of the following conditions:
\begin{enumerate}
\item the autocovariance kernels satisfy $\sum_{t \in \mathbb{Z}} \| r_t \|_2  < \infty$; 
\item %$(x,y) \mapsto r_t(x,y)$ is continuous on $\s \times \s$, $\forall t \in \mathbb{Z}$, and 
the autocovariance operators satisfy $\sum_{t\in \mathbb{Z}} \|\mathscr{R}_t \|_{\operatorname{TR}} < \infty$.
\end{enumerate}
\end{condition}

\begin{remark}
Notice that if $r_0(\cdot, \cdot)$ is continuous on $\s \times \s$, then each random field $\{T(x, t), \ x \in \s\}$,  $t \in \mathbb{Z},$ is mean-square continuous. 
Moreover, if we assume strong isotropy and stationarity, that is, $T(g\cdot, \cdot + \tau) \overset{d}{=}T(\cdot, \cdot), \ \forall g \in SO(3), \tau \in \mathbb{Z}$, continuity of all kernels $r_t(\cdot, \cdot), \ t \in \mathbb{Z},$ follows from mean-square continuity of $\{T(x, t), \ x \in \s\}, \ t \in \mathbb{Z};$ see \cite{marinucci2013mean}.
\end{remark}

%\begin{condition}\label{isotropy_condition}
%We assume that the autocovariance kernels and the autocovariance operators operators satisfy respectively:
%\begin{align}
%&\sum_{t \in \mathbb{Z}} \| r_t \|_\infty < \infty, \qquad \sum_{t \in \mathbb{Z}} \| \mathscr{R}_t \|_{\operatorname{TR}} < \infty. \label{isotropy::i1}
%\end{align}
%\end{condition}

In \cite{Panaretos2} there is an extensive discussion on the role of such assumptions. %, and at what cost it may be weakened. 
Similarly here, under Condition \ref{cond::summability} $(i)$, it is possible to define the spectral density kernel at frequency $\lambda \in [-\pi,\pi]$,
\begin{equation*}
f_\lambda(\cdot, \cdot) := \frac{1}{2\pi} \sum_{t \in \mathbb{Z} } e^{-i\lambda t} r_t(\cdot, \cdot),
\end{equation*}
where the convergence is in $\| \cdot \|_2$. It is uniformly bounded and also uniformly continuous in $\lambda$ with respect to $\| \cdot \|_2$. 
The spectral density operator $\mathscr{F}_\lambda: \LSc \to \LSc$, the operator induced by the spectral density kernel through right-integration, is self-adjoint and nonnegative definite for all $\lambda \in \mathbb{R}$.
Moreover, the following inversion formula holds in the $L^2$ sense:
\begin{equation}\label{inversion-formula}
\int_0^{2\pi} f_\alpha (\cdot, \cdot) e^{it\alpha} d\alpha = r_t(\cdot, \cdot).
\end{equation}

%It is uniformly bounded and also uniformly continuous in $\lambda$ with respect to  $\| \cdot \|_{L^2(\s \times \s)}$ and $\| \cdot \|_\infty$. For each $\lambda$, $f_\lambda(\cdot, \cdot)$ is continuous on $\s \times \s$ and
%\begin{equation*}
%f_{-\lambda}(x, y) = \overline{f_\lambda(x, y)}  = f_\lambda(y, x).
%\end{equation*}
%Moreover, for all $t \in \mathbb{Z}, \ x,y \in \s$, the following \textit{inversion formula} holds
%\begin{equation}\label{inversion-formula}
%\int_{-\pi}^\pi f_{\lambda} (x, y) e^{i\lambda t} d\lambda = r_{t}(x,y).
%\end{equation}

Under Condition \ref{cond::summability} $(ii)$, we can define the spectral density operator at frequency $\lambda \in [-\pi,\pi]$
\begin{align*}
& \mathscr{F}_\lambda := \frac{1}{2\pi} \sum_{t \in \mathbb{Z} } e^{-i\lambda t} \mathscr{R}_t,
\end{align*}
where the convergence holds in nuclear norm. 
$\mathscr{F}_\lambda$ is trace class and $\| \mathscr{F}_\lambda\|_{\operatorname{TR}}\le \sum_{t \in \mathbb{Z}} \| \mathscr{R}_t \|_{\operatorname{TR}}< \infty$, $\lambda \mapsto \| \mathscr{F}_\lambda\|_{\operatorname{TR}}$ is uniformly continuous and 
\begin{equation*}
\|\mathscr{F}_\lambda \|_{\operatorname{TR} }= \int_\s f_\lambda(x, x) d x.
\end{equation*}
The reader is referred to \cite{Panaretos2} for proofs of these assertions.

\begin{remark}
Condition \ref{cond::summability}  is strictly related to the concept of short memory stationary processes; indeed, stationary processes which exhibit short-range dependence are those with absolutely summable autocovariances and, hence, bounded and continuous spectral density, e.g., stationary ARMA processes. For functional time series, this translates into an "absolutely summable" autocovariance operators and a "bounded and continuous" spectral density operator, that is, their nuclear norms are, respectively, absolutely summable, and bounded and continuous.
\end{remark}

However, exploiting joint isotropy-stationarity of the space-time spherical random field, we can specialize all the previous results and obtain a neat expression for our quantities of interest. First of all, the sequence of zero-mean random coefficients satisfies
\begin{equation*}
\qquad \Ex [a_{\ell, m}(t) \overline{a_{\ell', m'}}(s)] = C_\ell(t-s) \delta_\ell^{\ell'} \delta_m^{m'}, \qquad t,s \in \mathbb{Z},
\end{equation*}
(see Equation \eqref{alm-stationarity} in Proof of Proposition \ref{theo:functional_cramer}) and, as a consequence of Schoenberg's Theorem \cite{schoenberg}, the covariance kernel is shown to have a spectral decomposition in terms of Legendre polynomials, i.e.,
\begin{equation}\label{gamma_space-time}
r_t(x, y)= \sum_{\ell=0}^\infty \frac{2\ell+1}{4\pi} C_\ell(t) P_{\ell}( \langle x, y \rangle ),
\end{equation}
where $\langle \cdot, \cdot \rangle$ denotes the standard inner product in $\mathbb{R}^3$, $P_\ell(\cdot)$ denotes the $\ell$-th Legendre polynomial \cite[Section 4.7]{szego} and the series is uniformly convergent.
\begin{remark}
Following the works \cite{schoenberg} and \cite{gneiting2}, in \cite{bergporcu} the authors give a mathematical characterization of covariance functions for isotropic stationary random fields over $\s \times \mathbb{R}$.
In \cite{clarke} the regularity properties of such covariance functions have been investigated for the case where a double Karhunen--Lo{\`e}ve expansion holds. Examples of random fields satisfying this decomposition are found in the Appendix of \cite{porcu}.
\end{remark}
As a consequence of \eqref{gamma_space-time}, 
%the orthonormal basis $\{\ylm, \ \ell \ge 0, \ m=-\ell,\dots,\ell\}$ satisfies
%\begin{align*}
%(\mathscr{F}_{\lambda} \ylm)(\cdot)&= \int_\s f_\lambda(\cdot, y) \ylm (y) d y  \nonumber \\
%&= \frac{1}{2\pi}  \sum_{t \in \mathbb{Z}}e^{-it\lambda}  \int_\s r_t(\cdot, y) \ylm (y) d y \nonumber\\
%&= \left [ \frac{1}{2\pi}  \sum_{t \in \mathbb{Z}} e^{-it\lambda}C_\ell (t) \right] \ylm(\cdot).
%\end{align*}
\begin{align*}
&f_\lambda(x,y)= \sum_{\ell=0}^{\infty}  \frac{2\ell+1}{4\pi} f_\ell(\lambda) P_\ell(\langle x, y \rangle ), %\\
%&\|\mathscr{F}_\lambda \|_{\operatorname{TR}}= \sum_{\ell=0}^\infty (2\ell+1) f_\ell(\lambda),
%&\Ex \| Z_\lambda \|^2 = (\lambda +\pi) \|\mathscr{F}_\lambda\|_{\operatorname{TR}},\\
%&\Ex \! \left  \| \int_{\pi}^\pi e^{i\lambda t}  d Z_\lambda \right \|^2 = 2\pi \|\mathscr{F}_\lambda \|_{\operatorname{TR}}.
\end{align*}
in $\| \cdot\|_2$ under Condition \ref{cond::summability} $(i)$
%and also
%\begin{align*}
%&r_t(x, y)= \sum_{\ell=0}^\infty \frac{2\ell+1}{4\pi}\left [ \int_{\pi} ^{\pi} e^{it\lambda} f_{\ell}(\lambda)d\lambda\right] P_{\ell}( \langle x, y \rangle ), \\
%&f_\lambda(x,y)= \sum_{\ell=0}^{\infty}  \frac{2\ell+1}{4\pi} \left [ \frac{1}{2\pi}  \sum_{t \in \mathbb{Z}} e^{-it\lambda}C_\ell (t) \right]  P_\ell(\langle x, y \rangle.
%\end{align*}
and in $\| \cdot \|_\infty$ under $(ii)$. It follows that $\mathscr{R}_t$ and $\mathscr{F}_\lambda$ satisfy
$$
\mathscr{R}_t \ylm = C_\ell(t) \ylm, \qquad \mathscr{F}_\lambda \ylm = f_\ell(\lambda) \ylm,
$$
that is, the $\ylm$'s are eigenfunctions of both $\mathscr{R}_t$ and $\mathscr{F}_\lambda$, and the $C_\ell(t)$'s and $f_\ell(\lambda)$'s are the associated eigenvalues. Moreover, by the inversion formula \eqref{inversion-formula}, we have that $f_\ell(\lambda) := \frac{1}{2\pi}  \sum_{t \in \mathbb{Z}} e^{-it\lambda} C_\ell (t)$.
%corresponds to the spectral density of the time series $\{a_{\ell, m}(t), \ t \in\mathbb{Z}\}$, for $m=-\ell,\dots,\ell$. 
%Indeed, Condition \ref{cond::summability} $(i)$ entails
%$$
%\sum_{t=-\infty}^\infty \left ( \sum_{\ell=0}^\infty (2\ell+1) C_\ell^2(t) \right)^{1/2} < \infty  
%$$
%Note that $\sum_{t \in \mathbb{Z}} |C_\ell(t)|^2 < \infty$ and $\sum_{\ell=0}^\infty (2\ell+1) |C_{\ell}(t)| < \infty$ (because of finite variance), but not necessarily $\sum_{t \in \mathbb{Z}} |C_\ell(t)| < \infty$, which means that, at the multipoles level, we are considering a broader class of processes than short memory processes.

The eigenvalues $f_\ell(\lambda)$ are also uniformly bounded and uniformly continuous in $\lambda$ with respect to $\| \cdot \|_2$. Indeed,
$$
0 \le f_\ell(\lambda) \le  \| f_\lambda \|_2 \le M, \quad  \text{for all } \lambda,
$$
and, given $\epsilon > 0$, there exists $\delta > 0$ such that
$$
|\lambda_1 - \lambda_2| > \delta \implies |f_\ell(\lambda_1) - f_\ell(\lambda_2)| \le \| f_{\lambda_1} -f_{\lambda_2} \|_2 < \epsilon.
$$
Clearly, under $(ii)$, the trace class norm is given by
$$
\|\mathscr{F}_\lambda \|_{\operatorname{TR}}= \sum_{\ell=0}^\infty (2\ell+1) f_\ell(\lambda).
$$

\section{Spectral representations}\label{sec::spectral-repr}

This section builds on the earlier works \cite{Panaretos2, Panaretos} and it provides some results on a double spectral representation, with respect to both the temporal and spatial components of the field.
The main purpose is to study these objects, trying to simultaneously capture the surface structure (spatial component) as well as the dynamics in time (temporal component); what in \cite{Panaretos} is called \emph{within/between curve dynamics}. We then specilize the results in \cite{Panaretos}, for dependent random functions on the interval $[0,1]$, to the case of the sphere, making this paper self-contained and complete.

The following proposition is the analogue of Theorem 2.1 in \cite{Panaretos} and it can be seen as the infinite-dimensional version of the well-known spectral representation of real-valued stationary processes.

\begin{proposition}[Spherical Functional Cram{\'e}r Representation]\label{theo:functional_cramer} Under Condition \ref{cond::summability} $(ii)$, $T_t$ admits the representation
\begin{equation}\label{cramer}
T_t=\int_{-\pi }^{\pi }e^{it\lambda }dZ_{\lambda }, \qquad \textnormal{a.s. in } L^2(\mathbb{S}^2),
\end{equation}
where, for fixed $\lambda$, $Z_\lambda$ is a random element of $\LSc$ with $\Ex \| Z_\lambda \|_2^2 = \sum_{\ell=0}^\infty (2\ell+1)  \int_{-\pi}^\lambda f_\ell (\nu) d \nu$, and the process $\{Z_\lambda,  -\pi \le \lambda \le \pi\}$ has orthogonal increments:
\begin{equation}
\Ex \left \langle Z_{\lambda_1} - Z_{\lambda_2}, Z_{\lambda_3}- Z_{\lambda_4} \right \rangle_2=0, \qquad \lambda_1  > \lambda_2 \ge \lambda_3 > \lambda_4.
\end{equation}
The representation \eqref{cramer} is called the Cram{\'e}r representation of $T_t$, and the stochastic integral involved can be understood as a Riemann-Stieltjes limit, in the sense that
\begin{equation*}
\Ex \left \| T_t - \sum_{j=1}^J e^{i\lambda_jt}(Z_{\lambda_{j+1}} - Z_{\lambda_{j}})\right \|_{\LSc}^2 \to 0, \qquad J \to \infty,
\end{equation*}
where $-\pi = \lambda_1 < \cdots < \lambda_{J+1}= \pi$ and $\max_{j=1,\dots,J} |\lambda_{j+1} - \lambda_j| \to 0$ as $J\to \infty$.
\end{proposition}

Now, we are going to establish a double spectral representation result, by showing the relation between the orthogonal increment process $\{\alpha_{\ell, m}(\lambda), -\pi \le \lambda\le \pi \}$ and $\{Z_\lambda, -\pi \le \lambda\le \pi \}$. It is worth to notice that, under Condition \ref{cond::summability}, all the results presented in \cite{Panaretos} can be easily extended to our framework, including the so-called Cram{\'e}r--Karhunen--Lo{\`e}ve Representation. Such a representation decomposes the space-time spherical random field into uncorrelated functional frequency components, exploiting an orthonormal basis for $\LSc$ made up of eigenfunctions of the spectral density operator $\mathscr{F}_\lambda$. 
However, in the anisotropic case, these eigenfunctions are unknown and have to be estimated.
The stronger conditions allows to apply directly theorems from \cite{Panaretos}, since we have an explicit eigenvalue-eigenfunction decomposition of the spectral density operator in terms of spherical harmonics. 

%The next result does not go through the eigenvalue-eigenfunction decomposition of $\mathscr{F}_\lambda$, but is is based on the standard orthonormal basis of spherical harmonics. 

\begin{proposition}[Spherical Cram{\'e}r--Karhunen--Lo{\`e}ve Representation]\label{theo:ckl} Under Condition \ref{cond::summability} $(ii)$, for every $t \in \mathbb{Z}$ and every $x \in \s$,
\begin{equation*}
\Ex \left | T(x,t) - \sum_{\ell=0}^L \sum_{m=-\ell}^\ell \int_{-\pi}^\pi e^{it\lambda} d\alpha_{\ell ,m}(\lambda)  \ylm (x) \right |^2 \to 0, \qquad L\to \infty,
\end{equation*}
with $\alpha_{\ell,m}(\lambda): = \langle Z_\lambda, Y_{\ell,m} \rangle_2$ and $Z_\lambda$ as defined in Proposition \ref{theo:functional_cramer},
$$
\mathbb{E}[\alpha_{\ell,m}(\omega) \alpha_{\ell',m'}(\beta)]=  \int_{-\pi}^{\min (\omega, \beta)}  f_\ell(\alpha) d \alpha \,\delta_\ell^{\ell'} \delta_{m}^{m'}.
$$
\end{proposition}

\begin{remark}
Note that the effective dimensionality of each frequency component is captured by the eigenvalues of the spectral density operators. The approximation error is then given by
\begin{equation*}
\Ex \left | T(x,t) - \sum_{\ell=0}^L \sum_{m=-\ell}^\ell \int_{-\pi}^\pi e^{it\lambda} d\alpha_{\ell ,m}(\lambda)  \ylm (x) \right |^2 = \sum_{\ell > L} \frac{2\ell + 1}{4\pi} \int_{-\pi}^\pi f_\ell(\lambda) d\lambda,
\end{equation*}
see also \cite[Remark 3.10]{Panaretos}.
\end{remark}

%\begin{enumerate}
%\item $\Ex|\alpha_{\ell m}(\lambda) |^2= \int_{-\pi}^\lambda f_{\ell}(\lambda) d\lambda;$
%\item $\Ex[ \alpha_{\ell m} (\lambda) \overline{ \alpha_{\ell m}} (\nu) ]=  \frac{\min(\lambda, \nu) + \pi}{2\pi} \sum_{t \in \mathbb{Z}} C_\ell(t), \ \qquad \lambda \ne \nu$;
%\item $\Ex (  \alpha_{\ell m} (\lambda) -  \alpha_{\ell m} (\nu))\overline{( \alpha_{\ell m}(\lambda) - \alpha_{\ell m} (\nu))}=\frac{\lambda - \nu}{2\pi} \sum_{t \in \mathbb{Z}} C_\ell(t), \qquad \lambda  > \nu.$
%\end{enumerate}

\section{Spherical functional ARMA}\label{sec::spharma}

In this section, we extend the spherical functional autoregressions (SPHAR), first introduced in \cite{cm19}, to include a moving-average term in the error. This leads to the definition of the so-called spherical functional autoregressive-moving average processes (SPHARMA). The main purpose here is to study the existence and uniqueness of an isotropic stationary solution of the functional autoregressive-moving average equation, see also \cite[Chapter 5]{bosq}.

We first recall the definitions of \emph{spherical white noise} and \textit{isotropic kernel operator}, see \cite{cm19, cdv19}. 

\begin{definition}
\label{SphericalWhiteNoise} The collection of random variables $ \lbrace Z(x,t), (x,t)\in {\mathbb{S}^{2}}\times \mathbb{Z} \rbrace $ is said to be a spherical white noise if:
\begin{enumerate}
\item for every fixed $t\in \mathbb{Z}$, $\{Z(x,t),\ x \in \s \}$ is a zero-mean
isotropic random field, with covariance kernel%
\begin{equation*}
r_{Z}(x,y)=\sum_{\ell =0}^{\infty }\frac{2\ell +1}{4\pi }C_{\ell
;Z}P_{\ell }(\left\langle x,y\right\rangle ), \quad \sum_{\ell
=0}^{\infty }\frac{2\ell +1}{4\pi }C_{\ell; Z}<\infty,
\end{equation*}%
$\left \lbrace  C_{\ell; Z}\right \rbrace$ denoting as usual the angular power spectrum of $Z(\cdot,t)$;
\item for every $t\neq s,$ $\Ex[a_{\ell,m;Z}(t) \overline{a_{\ell',m';Z}}(s)] = 0$, for all $\ell, \ell' \ge 0$, $|m|\le \ell$, $|m'|\le \ell'$, where
$$
a_{\ell,m;Z}(t) = \int_\s Z(x,t) \ylmc(x) dx.
$$
\end{enumerate}
We shall write $Z \sim SWN(0, \{C_{\ell; Z}\})$.
Moreover, $ \lbrace Z(x,t), (x,t)\in {\mathbb{S}^{2}}\times \mathbb{Z} \rbrace $ is said to be a strong spherical white noise if it satisfies $(i)$ and
the random fields $\{Z(x,t), \ x \in \s \}$, $t \in \mathbb{Z}$, are independent and identically distributed.
\end{definition}

%\begin{remark}
%Note that we are writing the spherical white noise as a collection of random variables defined on every pair $(x,t)\in \s\times \mathbb{Z}$. Alternatively, following \cite[page 72]{bosq}, one could give the definition in terms of random elements of a separable Hilbert space (in our case,
%corresponding to $\ls$). The two
%approaches are equivalent here, because in this paper we will always
%be dealing with jointly-measurable mean-square continuous random fields.
%\end{remark}
%
%\begin{remark}
%Note that we are defining the spherical white noise as a sequence of isotropic spherical random fields; hence, for every pair $(x,t)\in \s\times \mathbb{Z},$ $Z(x,t)$ is a well-defined random variable. Alternatively, one
%could view the noise as a sequence of random elements in a Hilbert space (in our case,
%corresponding to $\ls$), see \cite[page 72]{bosq}). The two
%approaches are equivalent here, because throughout this section and the next chapters we will always
%be dealing with jointly-measurable mean-square continuous random fields.
%\end{remark}
%
%\todo{
%\begin{remark}
%Note that we are defining the field as a collection of random variables
%defined on every pair $(x,t)\in \mathbb{S}^{2}\times \mathbb{Z}.$ This is not the
%approach considered in \cite{bosq}, where the fields are defined as random
%elements in a Hilbert space (in our case, corresponding to $L^{2}(\mathbb{S}%
%^{2})).$ However the two approaches are equivalent here, because we deal
%with jointly measurable mean-square continuous random fields.
%\end{remark}}

\begin{definition}\label{def::isotropic-kernel-op}
A spherical isotropic kernel operator is an application $\Phi
:L^{2}(\s)\rightarrow L^{2}(\s)$ which satisfies%
\begin{equation*}
(\Phi f)(x)=\int_{\s}k(\left\langle x,y\right\rangle )f(y)dy, \quad x \in \s,
\end{equation*}
for some continuous $k: [-1,1] \to \mathbb{R}.$
\end{definition}

%\todo{Let us recall that the operator norm of a kernel is given by%
%\begin{equation*}
%\sup_{f\in L^{2},\text{ }\left\Vert f\right\Vert
%_{L^{2}}=1}\int_{\s\times \s}f(y)k(\left\langle x,y\right\rangle
%)f(x)dxdx,
%\end{equation*}%
%or equivalently, by the maximum eigenvalue of $k(\left\langle
%.,.\right\rangle )$. 
%It follows that $\Phi $ has operator norm smaller than unity if and only if $%
%\sup_{\ell }|\phi _{\ell }|<1;$ the operator is Hilbert-Schmidt if $%
%\sum_{\ell }(2\ell +1)\phi _{\ell }^{2}<\infty ,$ and it is nuclear if $%
%\sum_{\ell }(2\ell +1)|\phi _{\ell }|<\infty $.}

The following representation holds in the $L^{2}$-sense for the kernel
associated with $\Phi$:
\begin{equation}
k(\left\langle x,y\right\rangle )=\sum_{\ell =0}^{\infty }\phi _{\ell }\frac{%
2\ell +1}{4\pi }P_{\ell }(\left\langle x,y\right\rangle ).
\label{kernel_expansion}
\end{equation}%
The coefficients $ \lbrace \phi _{\ell },\ \ell \geq 0 \rbrace $ corresponds to the eigenvalues of the
operator $\Phi $ and the associated eigenfunctions are the family of
spherical harmonics $\left \lbrace  Y_{\ell ,m}\right \rbrace $, yielding%
\begin{equation*}
\Phi Y_{\ell ,m}=\phi _{\ell }Y_{\ell ,m},
\end{equation*}%
Thus, it holds $\sum_{\ell }(2\ell +1)\phi _{\ell }^{2}<\infty $, and hence
this operator is Hilbert-Schmidt (see, e.g., \cite{Hsing}). In \cite{cm19, cdv19}, the authors also consider trace class operators, namely, such that $\sum_{\ell
}(2\ell +1)|\phi _{\ell }|<\infty$, for which the representation \eqref{kernel_expansion} holds pointwise for every $x, y \in \s$.

Now, we focus on a space-time spherical random field $\left \lbrace  T(x,t), (x,t) \in \mathbb{S}^{2} \times \mathbb{Z}\right \rbrace$, as defined in Section \ref{sec:intro}, for which it holds almost surely $T(\cdot,t) \in \LSr$, $t \in \mathbb{Z}$.

\begin{definition}
\label{SPHARMA(p,q)} $\left \lbrace  T(x,t), (x,t) \in \mathbb{S}^{2} \times \mathbb{Z}\right \rbrace $ is said to be a $\operatorname{SPHARMA}(p, q)$ process if there exist $p$
isotropic kernel operators $\left \lbrace  \Phi _{1},\dots ,\Phi _{p}\right \rbrace$, $q$
isotropic kernel operators $\left \lbrace  \Theta _{1},\dots ,\Theta _{q}\right \rbrace$ and a spherical white noise $\left \lbrace  Z(x,t), (x,t) \in \mathbb{S}^{2} \times \mathbb{Z}\right \rbrace $ such that
\begin{equation}
T(x, t)-\sum_{j=1}^p (\Phi _{j}T(\cdot, t-j))(x) = Z(x, t) + \sum_{j=1}^q (\Theta _{j}Z(\cdot, t-j))(x), \label{defar}
\end{equation}%
for all $(x,t)\in \mathbb{S}^{2}\times \mathbb{Z}$, the equality holding
both in the $L^{2}(\Omega )$ and in the $L^{2}( \s \times \Omega)$
sense.
\end{definition}

\begin{remark}
Note that, following \cite{cm19}, the solution process (as well as the spherical white noise) is defined pointwise, i.e., for
each $(x,t)$ there exists a random variable on $(\Omega ,\Im ,%
\mathbb{P})$ such that the identity \eqref{defar} holds. Alternatively, following \cite[page 72]{bosq}, one could give the definition in terms of random elements of $\ls$. The two approaches are actually equivalent, because we are dealing with jointly-measurable mean-square continuous random fields.
\end{remark}

\begin{remark}
Both \cite{cm19, cdv19} introduce two estimation procedures for the spherical autoregressive kernels $\{k_j, j = 1,\dots,p\}$ and investigate asymptotic properties of the corresponding nonparametric estimators.
Specifically, in \cite{cm19}, the authors focus on the solutions of a functional $L^2$-minimization problem, while, in \cite{cdv19}, they add a convex penalty term to study LASSO-type estimators under sparsity assumptions.
\end{remark}

Similarly to \cite{cm19, cdv19}, it is possible to write
\begin{equation}
a_{\ell ,m}(t) - \sum_{j=1}^p \phi _{\ell ;j}a_{\ell ,m}(t-j) =a_{\ell ,m;Z}(t) + \sum_{j=1}^q \itheta_{\ell; j} a_{\ell ,m;Z}(t-j); \label{arp}
\end{equation}%,
where the coefficients $\left \lbrace  \phi _{\ell;j}, \ \ell \ge 0, \ j=1,\dots,p\right \rbrace$ and $\left \lbrace  \itheta_{\ell;j}, \ \ell \ge 0,\ j = 1,\dots, q\right \rbrace$ are respectively the eigenvalues of the operators $\{\Phi_j, \ j=1,\dots,p\}$ and $\{\Theta_j, j=1,\dots,q\}.$ 

%To ensure identifiability, we assume that there exists at least $\ell$ and $\ell'$ such that $\phi _{\ell ;p}\neq 0$ and $\itheta_{\ell ;q}\neq 0$, see \cite{bosq, brockwelldavis}. 

Now, define the polynomials $\phi _{\ell }:\mathbb{C\rightarrow C}$ and $\itheta_{\ell }:\mathbb{C\rightarrow C}$, $\
\ell \geq 0$, such that
\begin{equation}
\phi _{\ell }(z)=1-\phi _{\ell ;1}z-\cdots -\phi _{\ell ;p}z^{p}, \qquad \itheta_{\ell }(z)=1 + \itheta_{\ell ;1}z-\cdots -\itheta_{\ell ;q}z^{q};
\label{AssoPoly}
\end{equation}
note that the actual degrees can change with $\ell$.  
Clearly, particular cases of the $\operatorname{SPHARMA}(p,q)$ process can be obtained by letting one of the two sequences constant and equal to $1$. For instance, if $\phi_\ell(z) \equiv 1,$ for all $\ell \ge 0$, we obtain a spherical functional moving-average process of order $q$ (or $\operatorname{SPHMA}(q)$),
whereas, If $\itheta_\ell(z) \equiv 1,$ for all $\ell \ge 0$, then we have the so-called spherical functional autoregressive process of order $p$ (or $\SPHAR(p)$), see \cite{cm19,cdv19}.

\begin{condition}[Causality/Stationarity]
\label{stationarity} The two sequences of polynomials in Equation \eqref{AssoPoly} are such that $\phi_\ell(\cdot)$ and $\itheta_\ell(\cdot)$ have no common zeroes and 
\begin{equation}
|z| \le 1 \ \Rightarrow \ \phi _{\ell }(z)\ne 0.
\end{equation}
More explicitly,
there are no roots in the unit disk, for all $\ell \geq 0$.
\end{condition}

\begin{remark}
We say that a $\operatorname{SPHARMA}(p,q)$ process satisfying Condition \ref{stationarity} is causal. Similarly, if 
$|z| \le 1 \ \Rightarrow \ \itheta _{\ell }(z)\ne 0,$ for all $\ell \ge 0$,
we call it invertible.
\end{remark}

\begin{remark}\label{rmk:infinf}
Since $\sum_{\ell }(2\ell +1)\phi _{\ell;j }^{2}<\infty $ for all $j=1,\dots,p$, Condition \ref{stationarity} actually ensures that there exists $\delta > 0$ such that $$|z| < 1 + \delta \ \Rightarrow \ \phi _{\ell }(z)\ne 0,\qquad \text{for all } \ell \geq 0.$$
Indeed, it is possible to show that if we consider all the polynomials $\phi_\ell(\cdot)$ of degree $d_\ell \ne 0$ with distinct roots $\xi_{\ell;1},\dots,\xi_{\ell; r_\ell}$, then
\begin{equation*}
|\xi_{\ell;j}| \ge \xi_\ast > 1, \qquad \text{uniformly over $\ell$}.
\end{equation*}
\end{remark}

\begin{condition}[Identifiability]
\label{Identifiability} The spherical white noise
process $\left \lbrace  Z(x,t), \ (x,t) \in \s \times \mathbb{Z}\right \rbrace  $ is such that $C_{\ell ;Z}>0$, for all $\ell
\ge 0.$
\end{condition}

%\begin{remark}
%The previous condition is an identifiability assumption; indeed, it is
%simple to verify from our arguments below that for $C_{\ell ;Z}=0$ the
%component of the kernel corresponding to the $\ell $-th multipole is not
%observable, i.e., the $\ARp$ process has the same distribution whatever the
%values of $\{\phi _{\ell;j}, j=1,\dots,p\}.$ It is possible, however, to estimate the
%"sufficient" version of the kernel, i.e., its projection on the relevant
%subspace, such that $C_{\ell ,Z}>0.$ The extension is straightforward and we
%avoid it just for brevity and notational simplicity. Of course, as a
%consequence we have that
%\begin{equation*}
%\int_{\mathbb{S}^{2} \times \mathbb{S}^2}\Gamma _{Z}(x,y)f(x)f(y)dxdy>0,\qquad \forall f(\cdot)\in
%L^{2}({\mathbb{S}^{2}}),\,  f(\cdot)\neq 0.
%\end{equation*}
%\end{remark}

Under this assumptions, Condition \ref{cond::summability} holds and the eigenvalues of the spectral density operators $\mathscr{F}_\lambda$ are defined as
$$
f_\ell(\lambda) = \frac{C_{\ell;Z}}{2\pi} \left | \frac{\itheta_\ell(e^{i\lambda})}{\phi_\ell(e^{i\lambda})} \right|^2, \qquad \lambda \in [-\pi,\pi].
$$

\begin{example}[$\SPHAR(1)$]\label{example::sphar1} The family of random variables $ \lbrace T(x,t), \ (x, t) \in \s \times \mathbb{Z}  \rbrace $ is a spherical autoregressive process of order one if for all pairs $(x,t)\in \s\times \mathbb{Z}$
it satisfies%
\begin{equation}\label{sphar1}
T(x, t)=(\Phi _{1}T(\cdot, t-1))(x)+Z(x, t);
\end{equation}%
in this case, Condition \ref{stationarity} simply becomes $|\phi _{\ell
}|< 1$, for all $\ell \geq 0$. Moreover,
$$
f_\ell(\lambda) = \frac{C_{\ell;Z}}{2\pi}\frac{1}{1-2\phi_\ell \cos \lambda + \phi_\ell^2},  \qquad \lambda \in [-\pi,\pi].
$$
\end{example}

The proof of the following statement is given already in \cite{bosq} for the simplest case of order one  Hilbert-valued autoregressive processes, but here we construct explicitly the solution with a slightly different argument for completeness.

\begin{proposition}\label{theo:solution}
Under Conditions \ref{stationarity} and \ref{Identifiability}, the unique isotropic stationary solution to \eqref{defar} is
given by%
\begin{equation}\label{eq::stationary-isotropic-solution}
T(x,t) = \lim_{k\rightarrow \infty }T_{k}(x,t), \qquad T_{k}(x,t)=\sum_{\ell
=0}^{L_k}\sum_{m=-\ell }^{\ell }\sum_{j=0}^{k}\psi_{\ell;j}a_{\ell
m;Z}(t-j)Y_{\ell m}(x),
\end{equation}
in the $L^2(\Omega)$ and $L^2(\s \times \Omega)$ sense. The coefficients $\{\psi_{\ell;j}\}$ are determined by the relation
\begin{equation}
\psi_\ell(z) = \sum_{j=0}^\infty \psi_{\ell;j} z^j = \itheta_\ell(z) / \phi_\ell (z), \qquad |z| \le 1.
\end{equation}
\end{proposition}

\begin{remark}
Notice that the isotropic stationary solutions of the $\SPHAR(1)$ equation \eqref{sphar1} take the form
\begin{equation*}
T(\cdot , t)=\sum_{j=0}^{\infty } \Phi _{1}^{j} Z(\cdot , t),
\end{equation*}
and Condition \ref{stationarity} is satisfied if and only if the operator norm $\left \|\Phi_1 \right \|_{\op} := \max_{\ell \ge 0} |\phi_\ell| < 1,$ see also \cite[Section 3.4]{bosq}.
\end{remark}

\section{Main results}\label{sec::main-results}

\subsection{SPHAR and SPHMA approximations of spectral density operators}

In what follows, we show that for any real-valued isotropic stationary random field, with spectral density kernels $f_\lambda$ satisfying Condition \ref{cond::summability} $(i)$, it is possible to find both a causal $\SPHAR(p)$ process and an invertible $\operatorname{SPHMA}(q)$ process whose spectral density kernels are arbitrarily close to $f_\lambda$ in the $L^2$ norm. This suggests that
the original process can be approximated in some sense by either a $\SPHAR(p)$ or a $\operatorname{SPHMA}(q)$
process. Similar results hold for the spectral density operator $\mathscr{F}_\lambda$ in the trace class norm under the stronger Condition \ref{cond::summability} $(ii)$.

Below we will denote with $T_L$ a \emph{band-limited} space-time spherical random field, namely such that it can be expanded in terms of finitely many spherical harmonics, up to a finite multipole $\ell = L$.
\begin{theorem}\label{theo::approx-ma} If $f_\lambda(\cdot, \cdot)$ is a spectral density kernel of an isotropic stationary process, satisfying Condition \ref{cond::summability} $(i)$, then $\forall \epsilon > 0$ there exists an invertible $\operatorname{SPHMA}(q)$ process
$$
T_L(x,t) = Z_L(x,t) + (\Theta_1 Z_L(\cdot,t-1))(x)+ \dots + (\Theta_q Z_L(\cdot,t-q))(x), \qquad Z_L \sim \text{SWN}(0, \{\sigma^2_\ell\}),
$$
with spectral density kernel $\tilde{f}_{\lambda}(\cdot, \cdot)$ such that 
$$
\|\tilde{f}_{\lambda} - f_\lambda\|_2 \le \epsilon \quad \text{for all } \lambda \in [-\pi,\pi],
$$
where $\sigma^2_\ell = (1+\itheta_{\ell;1}^2+\dots+\itheta_{\ell;q}^2)^{-1} \int_{-\pi}^{\pi} f_{\ell} (\lambda) d\lambda, \ \ell=0,\dots,L.$
Under only Condition \ref{cond::summability} $(ii)$, 
$$
\| \tilde{\mathscr{F}}_\lambda - \mathscr{F}_\lambda \|_{\operatorname{TR}}  \le \epsilon \quad \text{for all } \lambda \in [-\pi,\pi].
$$
\end{theorem}

\begin{theorem}\label{theo::approx-ar} If $f_\lambda(\cdot, \cdot)$ is a spectral density kernel of an isotropic stationary process, satisfying Condition \ref{cond::summability} $(i)$, then $\forall \epsilon > 0$ there exists a causal $\SPHAR(p)$ process
$$
T_L(x,t) = (\Phi_1 T_L(\cdot,t-1))(x)+ \dots + (\Phi_p T_L(\cdot,t-p))(x) + Z_L(x,t), \qquad Z_L \sim \text{SWN}(0, \{\sigma^2_\ell\}),
$$
with spectral density kernel $\tilde{f}_{\lambda}(\cdot, \cdot)$ such that 
$$
\|\tilde{f}_{\lambda} - f_\lambda\|_2 \le \epsilon \quad \text{for all } \lambda \in [-\pi,\pi].
$$
Under only Condition \ref{cond::summability} $(ii)$, 
$$
\| \tilde{\mathscr{F}}_\lambda - \mathscr{F}_\lambda \|_{\operatorname{TR}}  \le \epsilon \quad \text{for all } \lambda \in [-\pi,\pi].
$$
\end{theorem}

\begin{remark}
Even from an inferential point of view, these results suggest a way to estimate the spectral density operator $\mathscr{F}_\lambda$ of any isotropic stationary sphere-cross-time random field. Indeed, assuming to be able to observe the projections
of the fields on the spherical harmonics basis, it is possible to derive a sequence of rational estimators of the form
$$
\hat{f}_\ell(\lambda)  = \frac{\hat{\sigma}_\ell}{2\pi} \frac{|1 + \hat{\itheta}_{\ell ;1}e^{-i\lambda}-\cdots -\hat{\itheta}_{\ell ;q}e^{-i\lambda q}|^2}{|1-\hat{\phi}_{\ell ;1}e^{-i\lambda}-\cdots -\hat{\phi}_{\ell ;p}e^{-i\lambda p}|^2}, \qquad \ell \ge 0,
$$
for the eigenvalues characterizing $\mathscr{F}_\lambda$; see also \cite[Section 10.6]{brockwelldavis}.
\end{remark}

%\begin{theorem}\label{theo::approx-ar2} If $\mathscr{F}_\lambda$ is a spectral density operator of an isotropic stationary process, satisfying Condition \label{cond::summability}, then $\forall \epsilon > 0$ there exists a SPHAR($p$) process
%$$
%T_L(x,t) = \Phi_1 T_L(\cdot,t-1)+ \dots + \Phi_p T_L(\cdot,t-p) + Z_L(x,t)
%$$
%with spectral density operator $\tilde{\mathscr{F}}_\lambda$ such that 
%$$
%\| \tilde{\mathscr{F}}_\lambda - \mathscr{F}_\lambda \|_{\operatorname{TR}}  \le \epsilon \quad \text{for all } \lambda \in [-\pi,\pi].
%$$
%\end{theorem}

\subsection{The $L^2(\Omega)$ approximations}

Here we prove that any stationary spherical functional process can be approximated in the $L^2(\Omega)$ sense by both a SPHMA process and a SPHAR process of sufficiently large degree.
We also establish a functional Wold decomposition, which allows to represent the field as a sum of a linear process and a deterministic process, similarly to the finite-dimensional case, see \cite[Theorem 5.7.1]{brockwelldavis} and also \cite{bosq, BB:07}. This result is given in the auxiliary Lemma \ref{theo::wold} and it is instrumental for the proof of our last main theorem.

More formally, consider as usual a zero-mean isotropic stationary process $\{T(x,t), (x,t) \in \mathbb{S}^2\times \mathbb{Z}\}$. For each $(\ell,m)$, define the sequence of closed linear subspaces of $L^2(\Omega)$
%\begin{align*}
%&\mathcal{M}_n = \overline{\text{span}\{ T_t, \, -\infty < t \le n \}}, \qquad n \in \mathbb{Z}, \\
%&\mathcal{M}_{-\infty} = \bigcap_{n=-\infty}^\infty \mathcal{M}_n,
%\end{align*}
%and consider the one-step mean squared error 
%\begin{equation*}
%\mathbb{E} \| T_{n+1} - P_{\mathcal{M}_n} T_{n+1} \|^2 := \inf_{f \in \mathcal{M}_n} \mathbb{E} \| T_{n+1} - f \|^2 < \infty, 
%\end{equation*}
\begin{align*}
&\mathcal{M}_{\ell, m; n} = \overline{\text{span}\{ a_{\ell,m}(n), \, -\infty < t \le n \}}, \qquad n \in \mathbb{Z}, \\
&\mathcal{M}_{\ell, m; -\infty} = \bigcap_{n=-\infty}^\infty \mathcal{M}_{\ell,m;n},
\end{align*}
and the $\ell$-th one-step mean squared error
$$
\sigma_\ell^2 = \mathbb{E}| a_{\ell, m} (n) - P_{\mathcal{M}_{\ell,m;n}} a_{\ell, m} (n+1) |^2 =  \inf_{f \in \mathcal{M}_{\ell,m; n}}  \mathbb{E} | a_{\ell, m}(n+1) - f |^2 ,
$$
where $P_{\mathcal{M}_{\ell,m;n}}$ is the projection operator on $\mathcal{M}_{\ell, m; n}$, see \cite[Chapter 5]{brockwelldavis}.
Note that
$$
\sigma^2 = \sum_{\ell=0}^\infty (2\ell+1) \sigma_\ell^2 < \infty;
$$
indeed, if we define $\mathcal{M}_n = \overline{\text{span}\{ T_t, \, -\infty < t \le n \}}\subset \LSh$, we can observe that
$$
\infty > \mathbb{E} \| T_{n+1} - P_{\mathcal{M}_n} T_{n+1} \|_2^2 \ge \sum_{\ell, m}  \mathbb{E} | a_{\ell, m}(n+1) - P_{\mathcal{M}_{\ell, m;n}} a_{\ell, m} (n+1)  |^2.
$$

The following conditions will be used to state our second main result, with the additional Wold-like decomposition.
\begin{condition}\label{cond::wold}
Consider the following assumptions:
\begin{enumerate}
\item $ \sigma_\ell^2 > 0$ for all $\ell \ge 0$;
\item $\mathcal{M}_{\ell,m;-\infty} = \{0\}$  for all $\ell \ge 0,\ m=-\ell,\dots, \ell$.
\end{enumerate}
\end{condition}

\begin{remark}
Note that Conditions \ref{cond::wold} $(i)$ and $(i)$ entail that each stationary subprocess $\{a_{\ell,m}(t), t \in \mathbb{Z}\}$ is, respectively, a \emph{non-deterministic} and \emph{purely non-deterministic} process, see \cite[Section 5.7]{brockwelldavis}.
\end{remark}

\begin{lemma}[Wold Decomposition]\label{theo::wold}
An isotropic stationary random field $\{T(x,t), \ (x,t) \in \mathbb{S}^2 \times \mathbb{Z}\}$ satisfying Conditions \ref{cond::wold} $(i)$ can be expressed as 
$$
T(x,t) = \sum_{j=0}^\infty \Psi_j Z(x, t - j) + V(x,t),
$$
in $L^2(\Omega)$ and $L^2(\mathbb{S}^2 \times \Omega)$, where 
\begin{enumerate}
\item $\Psi_j Z(x, t - j) := \sum_{\ell, m} \psi_{\ell;j} a_{\ell, m; Z } (t-j) Y_{\ell,m}(x)$;
\item $\psi_{\ell;0} = 1$ and $\sum_{j=0}^\infty \psi_{\ell;j}^2 < \infty$ for all $\ell \ge 0$; %such that $\sigma^2_\ell \ne 0$; $\psi_{\ell;j} = 0, \ j\ge 0$, whenever $\sigma^2_\ell = 0$ 
\item $a_{\ell,m;Z}(t) \in \mathcal{M}_{\ell,m; t}$ for all $\ell,m,t$;
\item $Z \sim \text{SWN}$ with power spectrum $\{\sigma^2_\ell \}$;
\item $V(x,t) := \sum_{\ell, m} V_{\ell,m} (t) Y_{\ell,m} (x)$ with $V_{\ell,m} (t) \in \mathcal{M}_{\ell,m;-\infty}$ for all $\ell,m,t$;
\item $\Ex[a_{\ell,m;Z}(t)\overline{V_{\ell',m'}}(s)]=0$ for all $\ell,\ell',m,m',t,s.$
\end{enumerate}
\end{lemma}

\begin{remark}
The $h$-step prediction error $\sigma^2(h) := \sum_{\ell,m} \mathbb{E} | a_{\ell, m}(t+h) - P_{\mathcal{M}_{\ell, m;t}} a_{\ell, m} (t+h)  |^2$ is given by
$$
\sigma^2(h) = \sum_\ell (2\ell+1) \sigma^2_\ell \sum_{j=0}^{h-1} \psi_{\ell;j}^2.
$$
For a purely non-deterministic process it is clear that the $h$-step prediction mean squared error converges as $h \to \infty$ to the total variance of the process.
\end{remark}

\begin{remark}
For an isotropic stationary random field $\{T(x,t), \ (x,t) \in \mathbb{S}^2 \times \mathbb{Z}\}$ satisfying the hypotheses of Lemma \ref{theo::wold} and Condition \ref{cond::summability} $(i)$ or $(ii)$,
the eigenvalues of the spectral operator $\mathscr{F}_\lambda$ can be expressed as
$$
f_\ell(\lambda) = |\psi_\ell (e^{-i\lambda})|^2 \sigma_\ell^2/2\pi, \qquad \text{where } \psi_\ell (e^{-i\lambda}) = \sum_{j=0}^\infty \psi_{\ell;j} e^{-ij\lambda}.
$$
\end{remark}

We are now in the position to present our last main theorem. The statement makes precise the way
in which the $L^2(\Omega)$ approximations hold.

\begin{theorem}\label{theo::approx-l2}
An isotropic stationary random field $\{T(x,t), \ (x,t) \in \mathbb{S}^2 \times \mathbb{Z}\}$ satisfying Conditions \ref{cond::wold} $(i)$ and $(ii)$ is such that, for all $\epsilon > 0$, there exists integers $L$ and $q$ such that
\begin{equation}\label{approx1}
\mathbb{E}\left | T(x,t) -  Z(x, t) - \sum_{j=1}^q \Psi_{j;L} Z(x, t - j) \right |^2 \le \epsilon,
\end{equation}
where $\Psi_{j;L} Z(x, t - j) := \sum_{\ell=0}^L \sum_{m=-\ell}^\ell \psi_{\ell;j} a_{\ell, m; Z } (t-j) Y_{\ell,m}(x)$.

Moreover, for all $\epsilon > 0$, there exist integers $L$ and $p$ such that
\begin{equation}\label{approx2}
\mathbb{E}\left | T(x,t) -  \sum_{j=1}^p \Phi_{j;L} T(x, t - j) - Z(x,t) \right |^2 \le \epsilon,
\end{equation}
where $ \Phi_{j;L} T(x, t - j)  := \sum_{\ell=0}^L \sum_{m=-\ell}^\ell  \phi_{\ell;j} a_{\ell, m}(t) \ylm(x)$ for some coefficients $\{\phi_{\ell;j}\}$ such that $\sum_{j=1}^\infty \phi^2_{\ell;j} < \infty$. Both results also hold in the $L^2(\mathbb{S}^2 \times \Omega)$ sense.
\end{theorem}

\begin{remark}
Clearly, it is possible to obtain a $\operatorname{SPHARMA}(p,q)$ approximation, by first applying \eqref{approx2} to $\{T(x,t), \ (x,t) \in \mathbb{S}^2 \times \mathbb{Z}\}$ and then \eqref{approx1} to the residual $\{Z(x,t), \ (x,t) \in \mathbb{S}^2 \times \mathbb{Z}\}$.
\end{remark}

\section{Proofs}\label{sec::proofs}

\begin{proof}[Proof of Proposition \ref{theo:functional_cramer}]
The proof follows the same lines of \cite{Panaretos}. Let $\LSh$ be the Hilbert space of $\LSc$-valued random elements with finite second moment and $\mathbb{M}_0$ be the complex linear space spanned by all finite linear
combinations of the $T_t$'s,
\begin{equation*}
\mathbb{M}_0 := \{ \sum_{j=1}^n b_j T_j: \ n \in \mathbb{N}, \ b_j \in \mathbb{C}, \ t_j \in \mathbb{Z}   \} \subset \LSh.
\end{equation*}
Let $e_t: \nu \mapsto e^{it\nu}$, which belongs to the (complex) Hilbert space $L^2([-\pi, \pi], \|\mathscr{F}_{\nu} \|_{\operatorname{TR}} d \nu )$ endowed with the standard inner product
\begin{equation*}
\int_{-\pi}^\pi f(\nu) \overline{g(\nu)} \|\mathscr{F}_{\nu} \|_{\operatorname{TR}}d\nu, \qquad f,g \in L^2([-\pi, \pi], \|\mathscr{F}_{\nu} \|_{\operatorname{TR}} d\nu ),
\end{equation*}
$ \|\mathscr{F}_{\nu} \|_{\operatorname{TR}}$ being the nuclear norm of the spectral density operator.
Now, define the linear operator $E$ by linear extension of the mapping $T_t  \mapsto  e_t.$ $E$ is well defined and a linear isometry; in particular, the inversion formula \eqref{inversion-formula} gives
\begin{equation*}
\langle T_t, T_s \rangle_{\LSh} = \Ex \left[ \int_{\s} T(x,t) T(x, s) d x \right] = \int_{-\pi}^\pi e^{i(t-s)\nu} \|\mathscr{F}_{\nu} \|_{\operatorname{TR}} d\nu.
\end{equation*}
Then, we extend its domain to $\mathbb{M}$, the closure of $\mathbb{M}_0$ in $\LSh$ (see \cite{Panaretos} for further details); the extension has a well-defined inverse $E^{-1}: L^2([-\pi,\pi], \|\mathscr{F}_{\nu} \|_{\operatorname{TR}} d\nu) \to \mathbb{M}.$ For any $\omega \in (-\pi, \pi]$, we define $Z_{\omega} = E^{-1}(\mathbbm{1}_{[-\pi, \omega) }) \in \mathbb{M}$ and $Z_{-\pi } \equiv 0$. By the isometry property,
\begin{equation}\label{isometry_Z}
\langle Z_{\omega}, Z_{\beta} \rangle_{\LSh} =\langle E^{-1}\mathbbm{1}_{[-\pi, \omega)}, E^{-1}\mathbbm{1}_{[-\pi, \beta)} \rangle_{\LSh} = \int_{-\pi}^{\min\{\omega, \beta\}}  \|\mathscr{F}_{\nu} \|_{\operatorname{TR}} d\nu .
\end{equation}
Hence, $\omega \mapsto Z_{\omega}$ is an orthogonal increment process. 

The proof follows with definition of an operator $\zeta$ as extension of the mapping
\begin{equation*}
\sum_{j=1}^n g_j \mathbbm{1}_{[\omega_j, \omega_{j+1})} \mapsto \sum_{j=1}^n g_j (Z_{\omega_{j+1}} - Z_{\omega_j}).
\end{equation*}
The operator $\zeta$ is, by \eqref{isometry_Z}, an isomorphism with domain $ L^2([-\pi,\pi], \|\mathscr{F}_{\nu} \|_{\operatorname{TR}} d\nu)$, and in addition $\zeta = E^{-1}$. This in turn implies $T_t = E^{-1}(e_t)= \zeta(e_t)$. If $g$ is cadlag with a finite number of jumps, then $\zeta(g)$ is in fact the Riemann-Stieltjes integral (in the mean square sense) with respect to the orthogonal increment process $Z_\omega$:
\begin{equation*}
\zeta(g) = \int_{-\pi}^\pi g(\lambda) dZ_\lambda.
\end{equation*}
In conclusion, $T_t=\int_{-\pi }^{\pi }e^{it\lambda }dZ_{\lambda },$ as claimed.
\end{proof}

\begin{proof}[Proof of Proposition \ref{theo:ckl}]
Define $a_{\ell, m}(t):=\langle T_t, \ylm \rangle_2$. For every fixed $(\ell,m)$, $\{a_{\ell, m}(t), \ t \in \mathbb{Z}\}$ forms a zero-mean complex-valued stationary sequence, i.e.,
\begin{align*}
&\Ex[a_{\ell, m}(t)]=0;\\
&\Ex|a_{\ell, m}(t)|^2 < \infty;\\
&\Ex[a_{\ell ,m}(t) \overline{a_{\ell, m}}(s)]= C_{\ell}(t-s).
\end{align*}

Indeed, by Fubini's Theorem we have
\begin{equation*}
\Ex[a_{\ell, m}(t)] = \Ex \langle T_t, \ylm \rangle_2 = \langle 0, \ylm \rangle_2 = 0.
\end{equation*}
Moreover, 
\begin{equation*}
\Ex|a_{\ell, m}(t)|^2 \le \Ex\| T_t \|_2^2 < \infty,
\end{equation*}
and, again by Fubini's Theorem,
\begin{equation}\label{alm-stationarity}
\Ex[a_{\ell, m}(t) \overline{a_{\ell', m'}}(s)] =  \Ex\left [ \langle T_t, \ylm \rangle_2 \overline{\langle T_s, \ylm \rangle_2} \right] = \langle \mathscr{R}_{t-s} \ylm, Y_{\ell',m'} \rangle_2 \, \delta _{\ell}^{\ell'}  \delta_{m}^{m'}.
\end{equation}
Therefore, as a result of the Spectral Theorem for stationary time series (see for instance \cite{brockwelldavis}), the following representation holds
\begin{equation*}
a_{\ell, m}(t) = \int_{-\pi}^\pi e^{i\lambda t} d \alpha_{\ell, m}(\lambda), \qquad \textnormal{a.s.},
\end{equation*}
where $\{\alpha_{\ell, m}(\lambda), \ -\pi \le \lambda \le \pi\}$ is an orthogonal increment process, and the stochastic integral involved can be understood as a Riemann-Stieltjes limit, in the sense that
\begin{equation*}
\Ex \left | a_{\ell, m}(t) - \sum_{j=1}^J e^{i\lambda_jt} \left [\alpha_{\ell ,m}(\lambda_{j+1}) - \alpha_{\ell, m}(\lambda_j) \right ] \right |^2 \to 0, \qquad J \to \infty,
\end{equation*}
where $-\pi = \lambda_1 < \cdots < \lambda_{J+1}= \pi$ and $\max_{j=1,\dots,J} |\lambda_{j+1} - \lambda_j| \to 0$ as $J\to \infty$

Moreover, recall from the Spectral Theorem for isotropic random fields on $\s$ that 
\begin{equation*}
\Ex \left  | T(x,t) - \sum_{\ell=0}^L \sum_{m=-\ell}^\ell \int_{-\pi}^\pi e^{i\lambda t} d \alpha_{\ell ,m}(\lambda) \ylm \right |^2   \to 0, \qquad L\to \infty.
\end{equation*}
Now we prove that $\alpha_{\ell ,m}(\lambda) \overset{\textnormal{a.s.}}{=} \langle Z_\lambda, \ylm \rangle_2$, $Z_\lambda$ as defined in Proposition \ref{theo:functional_cramer}

For a fixed $\lambda$, $\alpha_{\ell, m}(\lambda) \in \overline{\textnormal{span}\{a_{\ell, m} (t), \ t \in \mathbb{Z}\}}=\overline{\textnormal{span}\{\langle T_t, \ylm \rangle_2, \ t \in \mathbb{Z}\}} \subset L^2(\Omega)$. Indeed, from \cite{brockwelldavis}, we know that there exist a sequence $\{\alpha_j\}_{j \in \mathbb{Z}} \subset \mathbb{C}$ such that 
\begin{equation*}
\Ex \left | \alpha_{\ell ,m} (\lambda)- \sum_{|j| \le k} \alpha_j \langle T_{t_j}, \ylm \rangle_2 \right |^2 \to 0, \qquad k \to \infty.
\end{equation*}
The sequence is given by
\begin{equation}\label{fourier_I}
\alpha_j = \frac{1}{2\pi} \int_{-\pi}^\pi \mathbbm{1}_{[-\pi, \lambda)}(\nu) e^{-ij\nu} d\nu, \qquad j \in \mathbb{Z}.
\end{equation}
Now,
\begin{align*}
\Ex \left | \langle Z_\lambda, \ylm \rangle_2 - \sum_{|j|\le k} \alpha_j \langle T_{t_j}, \ylm \rangle_2 \right |^2&=  \Ex \left | \left \langle Z_\lambda - \sum_{|j|\le k} \alpha_j T_{t_j}, \ylm \right \rangle_2 \right |^2 \nonumber \\
&\le  \Ex \left \| Z_\lambda - \sum_{|j|\le k} \alpha_j T_{t_j} \right \|_{\LSc}^2 ,
\end{align*}
by Cauchy-Schwartz inequality e orthonormality of the $\ylm$'s.

We just need to prove that 
\begin{equation*}
 \Ex \left \| Z_\lambda - \sum_{|j|\le k} \alpha_j T_{t_j} \right \|_{\LSc}^2 \to 0, \qquad k \to \infty.
\end{equation*}
Recall that $\{	\alpha_j,  \ j \in \mathbb{Z}\}$ as defined in \eqref{fourier_I} represent the Fourier coefficients of the indicator function $\mathbbm{1}_{[-\pi, \lambda)}(\cdot)$. Then, its $k$-th order Fourier series approximation is given by
\begin{equation*}
h_k(\cdot) = \sum_{|j| \le k} a_j e^{ik \cdot},
\end{equation*}
and $ \sum_{|j|\le k} \alpha_j T_{t_j} = E^{-1}(h_k)$, where $E$ is the isomorphism of Proposition \ref{cramer}.
Since $\|\mathscr{F}_\lambda \|_{\operatorname{TR}} \le const$ uniformly over $\lambda$ by assumption, it holds that 
\begin{equation*}
\int_{-\pi}^\pi |h_k(\nu) -  \mathbbm{1}_{[-\pi, \lambda)}(\nu)|^2  \| \mathscr{F}_\nu \|_{\operatorname{TR}} d\nu \to 0, \qquad k \to \infty.
\end{equation*}
By continuity of $E$, we conclude that 
\begin{equation*}
E^{-1}(h_k) \to E^{-1}( \mathbbm{1}_{[-\pi, \lambda)} ) = Z_\lambda, \qquad k \to \infty,
\end{equation*}
in the $L^2$-sense.

Moreover, see \cite{Panaretos},
$$
\mathbb{E}[Z_\omega(x) \overline{Z_{\beta}(y)}] = \int_{-\pi}^{\min (\omega, \beta)} f_\alpha(x, y) d \alpha, \qquad \text{a.e.}.
$$
Under isotropy this is equal to
$$
\mathbb{E}[Z_\omega(x) \overline{Z_{\beta}(y)}]   = \sum_{\ell=0}^\infty \frac{2\ell+1}{4\pi}   \int_{-\pi}^{\min (\omega, \beta)} f_\ell(\alpha) d \alpha \, P_\ell(\langle x,y \rangle),  \qquad \text{a.e.},
$$
which means that
$$
\mathbb{E}[\alpha_{\ell,m}(\omega) \alpha_{\ell',m'}(\beta)]=  \int_{-\pi}^{\min (\omega, \beta)}  f_\ell(\alpha) d \alpha \,\delta_\ell^{\ell'} \delta_{m}^{m'}.
$$
Note that 
$$
0 \le  \int_{-\pi}^{\min (\omega, \beta)} f_\ell(\alpha) d \alpha \le \int_{-\pi}^{\pi} f_\ell(\alpha) d \alpha = C_\ell(0).
$$

%Now, define the operator 
%\begin{align*}
%\ylm \otimes \ylm&: \LSc \to \LSc\\
%&: f \mapsto \langle f, \ylm \rangle \ylm.
%\end{align*}
%Note that
%\begin{equation*}
%\int_{-\pi}^{\pi} e^{it\lambda}  \ylm \otimes \ylm d Z_\lambda = \int_{-\pi}^{\pi} e^{it\lambda} d \alpha_{\ell, m}(\lambda) \ylm, \qquad \textnormal{a.s. in } L^2,
%\end{equation*}
%see \cite{Panaretos} for a definition of stochastic integrals of operators as the one on the left-hand side. 
%By linearity of the stochastic integral in \eqref{eq::spectral2} (see \cite{Panaretos}), we conclude the proof.
\end{proof}

The next proof is composed of two steps. First we show that the $T(x,t)=\lim_{k\rightarrow \infty }T_{k}(x,t)$ is a solution of the $\operatorname{SPHARMA}(p,q)$ equation \eqref{defar}; and then we prove that any isotropic stationary solution of \eqref{defar} takes the form \eqref{eq::stationary-isotropic-solution}.

\begin{proof}[Proof of Proposition \ref{theo:solution}]
First note that, under Condition \ref{stationarity}, for any $\ell \ge 0$, $\itheta_\ell(z)/\phi_\ell(z)$ has a power series expansion, that is,
\begin{equation*}
\itheta_\ell(z) / \phi_\ell (z) = \sum_{j=0}^\infty \psi_{\ell;j} z^j =  \psi_\ell(z), \qquad |z| \le 1,
\end{equation*}
and
\begin{equation*}
\sum_{j=0}^\infty |\psi_{\ell;j}| < \infty,
\end{equation*}
see \cite[Proof of Theorem 3.1.1, page 85]{brockwelldavis}.

Now, let us show first that the sequence $\{T_{k}\}$ is Cauchy. Indeed we have, for $%
k^{\prime }>k, L_{k'} >L_k,$%
\begin{align*}
T_{k^{\prime }}(x,t)-T_{k}(x,t) &=\sum_{\ell =0}^{L_k}\sum_{m=-\ell
}^{\ell }\sum_{j=k}^{k^{\prime }}\psi_{\ell;j}a_{\ell, m;Z}(t-j)\ylm(x) \\
&+\sum_{\ell =L_k+1}^{L_{k^{\prime }}}\sum_{m=-\ell }^{\ell
}\sum_{j=0}^{k^{\prime }}\psi_{\ell;j}a_{\ell, m;Z}(t-j)\ylm(x)
\end{align*}%
and, therefore,%
%\begin{align}
%&\mathbb{E}\left[ \int_\s|T_{k^{\prime }}(x,t)-T_{k}(x,t)|^2 dx\right]\notag \\
%=&\sum_{\ell =0}^{L_k}(2\ell+1)C_{\ell ;Z}\sum_{j=k}^{k^{\prime }}\psi_{\ell;j}^{2} +\sum_{\ell =L_k+1}^{L_{k^{\prime }}}(2\ell+1)C_{\ell ;Z} \sum_{j=0}^{k^{\prime }}\psi _{j;\ell
%}^{2} \notag \\
%\leq&\sum_{\ell =0}^{\infty }(2\ell+1)C_{\ell ;Z}\sum_{j=k}^{\infty }\psi_{\ell;j}^{2} + \sum_{\ell =L_k+1}^{\infty}(2\ell+1)C_{\ell ;Z}\sum_{j=0}^{\infty}\psi _{j;\ell
%}^{2} . \label{ineq}
%\end{align}
\begin{align}
&\mathbb{E}\left\vert T_{k^{\prime }}(x,t)-T_{k}(x,t)\right\vert ^{2} \notag \\
=&\sum_{\ell =0}^{L_k}\frac{%
2\ell +1}{4\pi }C_{\ell ;Z}\sum_{j=k}^{k^{\prime }}|\psi_{\ell;j}|^{2} +\sum_{\ell =L_k+1}^{L_{k^{\prime }}}\frac{%
2\ell +1}{4\pi }C_{\ell ;Z} \sum_{j=0}^{k^{\prime }}|\psi _{j;\ell
}|^{2} \notag \\
\leq&\sum_{\ell =0}^{\infty }\frac{%
2\ell +1}{4\pi }C_{\ell ;Z}\sum_{j=k}^{\infty }|\psi_{\ell;j}|^{2} + \sum_{\ell =L_k+1}^{\infty}\frac{%
2\ell +1}{4\pi }C_{\ell ;Z}\sum_{j=0}^{\infty}|\psi _{j;\ell
}|^{2} . \label{ineq}
\end{align}
For $\ell \ge 0,$ consider the stationary process
\begin{equation*}
X_{\ell }(t)=\sum_{j=0}^{\infty }\psi_{\ell;j}\varepsilon (t-j), \qquad t \in \mathbb{Z};
\end{equation*}%
here we take $\left\lbrace  \varepsilon (t), \ t \in \mathbb{Z}\right\rbrace  $ to be a white noise sequence
with variance identically equal to one. The spectral density of $\{X_{\ell }(t), \ t\in \mathbb{Z}\}$ is given by (see \cite{brockwelldavis})%
$$
f_{\ell }(\lambda ) := \frac{1}{2\pi }\left\vert \sum_{j=0}^{\infty }\psi
_{j;\ell }\exp (i\lambda j)\right\vert ^{2} = \frac{1}{2\pi }|\psi_\ell
(e^{i\lambda })|^{2} = \frac{1}{2\pi }\left \vert \frac{\itheta_\ell (e^{i\lambda})}{\phi
_{\ell }(e^{i\lambda })} \right\vert ^{2}, 
$$
Now, recall the identity%
\begin{equation*}
\Var[X_{\ell }(t)]=\sum_{j=0}^{\infty }|\psi_{\ell;j}|^{2}=\int_{-\pi }^{\pi
}f_{\ell }(\lambda )d\lambda.
\end{equation*}%
whence%
\begin{equation*}
\sum_{j=0}^{\infty }|\psi_{\ell;j}|^{2}=\frac{1}{2\pi }\int_{-\pi }^{\pi
}|\psi_\ell (e^{i\lambda })|^{2\text{ }}d\lambda =\frac{1}{%
2\pi }\int_{-\pi }^{\pi }\left \vert \frac{\itheta_\ell (e^{i\lambda})}{\phi
_{\ell }(e^{i\lambda })} \right\vert ^{2}%
d\lambda .
\end{equation*}%
Moreover, under Condition \ref{stationarity}, for the non-degenerate polynomials it holds that
\begin{equation*}
|\phi _{\ell }(e^{i\lambda })| = \prod_{u=1}^{r_\ell} |1-\xi^{-1}_{\ell; u} e^{i\lambda}|^{s_{\ell;u}} \ge \prod_{u=1}^{r_\ell}( 1-|\xi^{-1}_{\ell; u}| )^{s_{\ell;u}} \ge ( 1-\xi^{-1}_{\ast} )^p >0\ ,
\end{equation*}
see Remark \ref{rmk:infinf}; hence, as a consequence,
\begin{equation*}
\sum_{j=0}^{\infty }|\psi_{\ell;j}|^{2} =\frac{1}{2\pi }\int_{-\pi }^{\pi }%
\left | \frac{\itheta_\ell(e^{i\lambda})}{\phi _{\ell }(e^{i\lambda })} \right |^2d\lambda \leq const \left (\frac{\xi_\ast }{%
\xi_\ast -1}\right )^{2p} 
<const,
\end{equation*}%
uniformly over $\ell$, and
\begin{equation*}
\sum_{\ell =0}^{\infty }\sum_{j=0}^{\infty }|\psi_{\ell;j}|^{2}\frac{2\ell +1%
}{4\pi }C_{\ell ;Z}<\infty.
\end{equation*}%
Then, by the Dominated Convergence Theorem, we have 
\begin{equation*}
\lim_{k\rightarrow \infty }\sum_{\ell =0}^{\infty }\sum_{j=k}^{\infty }|\psi
_{j;\ell }|^{2}\frac{2\ell +1}{4\pi }C_{\ell ;Z}=\sum_{\ell =0}^{\infty
}\left\lbrace  \lim_{k\rightarrow \infty }\sum_{j=k}^{\infty }|\psi _{j;\ell
}|^{2}\right\rbrace  \frac{2\ell +1}{4\pi }C_{\ell ;Z}=0,
\end{equation*}%
and \eqref{ineq} $\rightarrow 0$ as $k\rightarrow \infty $, so that $\left\lbrace  T_{k}\right\rbrace  $ is indeed a Cauchy
sequence. The proof that it satisfies \eqref{defar} is standard; without loss of generality here we consider the $\SPHAR(p)$ case. Hence, we have that%
\begin{align*}
&\Ex \left | T(x,t)-\sum_{j=1}^p (\Phi _{j} T(\cdot,t-j))(x) - Z(x,t)\right |^2 \\
=&\lim_{k\rightarrow \infty }\Ex \left |  T_{k}(x,t)-\sum_{j=1}^p (\Phi _{j} T_k(\cdot,t-j))(x)-Z(x,t)\right |^2 \\
=&\lim_{k\rightarrow \infty }\Ex \left |  \sum_{j=1}^p \sum_{\ell =1}^{L_k}\sum_{m=-\ell }^{\ell }\sum_{h=k-j+1}^{k}\phi
_{\ell;j}\psi _{h;\ell }a_{\ell, m;Z}(t-j-h)\ylm (x)\right |^2\\
\leq & \sum_{j=1}^p \lim_{k\rightarrow \infty } \Ex \left |  \sum_{\ell =1}^{L_k}\sum_{m=-\ell }^{\ell }\sum_{h=k-j+1}^{k}\phi
_{\ell;j}\psi _{h;\ell }a_{\ell, m;Z}(t-j-h)\ylm (x)\right |^2,
\end{align*}%
which again is easily shown to be zero by $\lim_{k\rightarrow \infty }\sum_{h=k-j+1}^\infty |\psi_{h;\ell}|^2=0$ and Dominated Convergence Theorem. The argument involving the $L^2(\s \times \Omega)$ limit is analogous.

\medskip
To complete the proof, we need to show that if $\{U(x,t), (x,t) \in \s \times \mathbb{Z}\}$ is an isotropic stationary solution of \eqref{defar}, then we must have $$U(x,t) = \sum_{\ell=0}^\infty \sum_{m=-\ell}^\ell \sum_{j=0}^\infty \psi_{j; \ell} a_{\ell, m; Z}(t-j)\ylm(x)$$ in $L^2(\mathbb{S}^2 \times\Omega)$ and $L^2(\Omega)$.
If $\{U(x,t), (x,t) \in \s \times \mathbb{Z}\}$ is an isotropic stationary solution of \eqref{defar}, then $a_{\ell, m; U}(t) = \int_{\mathbb{S}^2} U(x,t)\ylm(x) dx$ is a stationary solution of the standard $\ARp$ equation and, under Condition \ref{stationarity}, $$a_{\ell, m; U}(t)= \sum_{j=0}^\infty \psi_{\ell;j} a_{\ell, m; Z} (t-j), \qquad \text{in } L^2(\Omega).$$
Then, by stationarity and isotropy, $\mathbb{E} |a_{\ell, m; U}(t)|^2 = C_{\ell;U}$ and
\begin{align*}
& \Ex \left \| \sum_{\ell=0}^{L_k} \sum_{m=-\ell}^\ell a_{\ell, m; U} (t) \ylm - \sum_{\ell =0}^{L_k}\sum_{m=-\ell }^{\ell
}\sum_{j=0}^{k}\psi_{\ell;j}a_{\ell, m; Z}(t-j)\ylm \right \|_{\LSc}^2 \\
=& \sum_{\ell=0}^{L_k} \sum_{j=k+1}^\infty |\psi_{\ell;j}|^2 (2\ell + 1)  C_{\ell;U},
\end{align*}
which goes to zero as $k \to \infty$. Hence, by triangular inequality,
\begin{equation*}
\Ex \left \| U(\cdot,t) - \sum_{\ell =0}^{L_k}\sum_{m=-\ell }^{\ell
}\sum_{j=0}^{k}\psi_{\ell;j}a_{\ell, m; Z} (t-j) \ylm \right \|_{\LSc}^2
\to 0, \qquad k \to \infty. 
\end{equation*}
The same result holds in the sense of convergence in $L^2(\Omega)$, for every fixed pair $(x,t)$. Indeed, we have
\begin{align*}
&\mathbb{E} \left | \sum_{\ell=0}^{L_k} \sum_{m=-\ell}^\ell a_{\ell, m; U} (t) \ylm (x) - \sum_{\ell =0}^{L_k}\sum_{m=-\ell }^{\ell
}\sum_{j=0}^{k}\psi_{\ell;j}a_{\ell, m; Z}(t-j)\ylm(x) \right |^2 \\
=& \sum_{\ell=0}^{L_k}  \sum_{j=k+1}^\infty| \psi_{\ell;j}|^{2}   \frac{2\ell + 1}{4\pi}C_{\ell;U}.
\end{align*}
\end{proof}

\begin{proof} [Proof of Theorem \ref{theo::approx-ma}]
First observe that, under Condition \ref{cond::summability} $(i)$, $\forall \epsilon > 0$ there exists $L$ such that
$$
\sup_{\lambda \in [-\pi, \pi]} \left(  \sum_{\ell >L}  (2\ell+1)|f_\ell(\lambda)|^2 \right)^{1/2} \le \epsilon/2.
$$
Indeed, $\| f_\lambda\|_2 \le \sum_{t \in \mathbb{Z}} \| r_t \|_2$, then
$$
\left ( \sum_{\ell >L}  (2\ell+1)|f_\ell(\lambda)|^2 \right)^{1/2} \le \sum_{t=-\infty}^\infty \left (\sum_{\ell > L} (2\ell+1) C_\ell^2(t)\right )^{1/2} \le \sum_{t=-\infty}^\infty \left ( \sum_{\ell=0}^\infty (2\ell+1) C_\ell^2(t) \right)^{1/2} < \infty.
$$
From \cite[Theorem 4.4.3 and Corollary 4.4.1]{brockwelldavis}, take $\frac{\epsilon}{2(L+1)^2}$, then there exists an MA($q_\ell$) process with white noise variance $\sigma^2_\ell = (1+\itheta_{\ell;1}^2+\dots+\itheta_{\ell;q}^2)^{-1} \int_{-\pi}^{\pi} f_{\ell} (\lambda) d\lambda$ and spectral density $\tilde{f}_{\ell}$ such that
$$
| \tilde{f}_\ell(\lambda) - f_\ell(\lambda)| \le \frac{\epsilon}{2(L+1)^2}, \qquad \text{for all $\lambda \in [-\pi,\pi]$}.
$$
Now, define 
$$
\tilde{f}_\lambda(\langle x,y \rangle) =\sum_{\ell=0}^L \tilde{f}_\ell (\lambda) \frac{2\ell+1}{4\pi} P_\ell(\langle x,y \rangle),
$$ 
and $\tilde{\mathscr{F}}_\lambda$ the operator induced by right integration. Take $q = \max_{\ell \le L} q_\ell$. 
Then, we have
$$
\|\tilde{f}_{\lambda} - f_\lambda\|_2 \le \left( \sum_{\ell=0}^L (2\ell+1)| \tilde{f}_\ell(\lambda) - f_\ell(\lambda)|^2 \right)^{1/2}+ \left( \sum_{\ell >L}  (2\ell+1)|f_\ell(\lambda)|^2\right)^{1/2} \le \epsilon,
$$
uniformly over $\lambda  \in [-\pi,\pi]$.
%$f_\lambda(\cdot,\cdot)$ is continuous on $\mathbb{S}^2 \times \mathbb{S}^2$ since $r_t(\cdot,\cdot)$ is continuous. Hence, by Mercer Theorem,
%$$
%\sum_{\ell >L}  (2\ell+1)f_\ell(\lambda) \to 0
%$$
%not uniformly in $\lambda$.

Similarly, under Condition \ref{cond::summability} $(ii)$,
$$
\| \tilde{\mathscr{F}}_\lambda - \mathscr{F}_\lambda \|_{\operatorname{TR}}  \le \sum_{\ell=0}^L (2\ell+1)| \tilde{f}_\ell(\lambda) - f_\ell(\lambda)| + \sum_{\ell >L}  (2\ell+1)f_\ell(\lambda) \le \epsilon,
$$
for all $\lambda \in [-\pi,\pi]$.
\end{proof}

\begin{proof} [Proof of Theorem \ref{theo::approx-ar}] This proof follows the same lines of the previous one. Here, we make use of \cite[Theorem 4.4.3 and Corollary 4.4.2]{brockwelldavis}, which ensure that  for $\ell= 0,\dots,L$ there exists an AR($p_\ell$) process with spectral density $\tilde{f}_{\ell}$ such that
$$
| \tilde{f}_\ell(\lambda) - f_\ell(\lambda)| \le \frac{\epsilon}{2(L+1)^2}, \qquad \text{for all $\lambda \in [-\pi,\pi]$}.
$$
We then take $p = \max_{\ell \le L} p_\ell$.
\end{proof}

\begin{proof}[Proof of Lemma \ref{theo::wold}]
We just prove convergence in the $L^2(\Omega)$ sense; the $L^2(\s \times \Omega)$ argument is analogous. 
First of all, by Theorem 5.7.1 in \cite{brockwelldavis}, we can deduce that, for fixed $(\ell, m)$,
$$
a_{\ell, m}(t) = 
\sum_{j=0}^\infty \psi_{\ell;j} a_{\ell, m; Z} (t-j) + V_{\ell,m}(t), \qquad \text{in } L^2(\Omega),
$$
and
$$
\infty > \sum_\ell (2\ell+1) C_{\ell}(0) =\sum_{\ell} (2\ell+1) \sigma^2_\ell \sum_{j=0}^\infty \psi^2_{\ell;j} + \sum_\ell (2\ell+1) v_{\ell},
$$
were $v_\ell = \mathbb{E}|V_{\ell,m}(t)|^2$. In addition, since $a_{\ell, m;Z} (t) \in \mathcal{M}_{\ell,m;t}$ and $V_{\ell,m}(t) \in \mathcal{M}_{\ell,m;-\infty}$, it is clear that $\mathbb{E}[a_{\ell, m;Z} (t) \overline{a_{\ell ', m';Z}} (s)] = \sigma^2_\ell  \delta_t^s\delta_{\ell}^{\ell'} \delta_m^{m'}$, and $\Ex[a_{\ell,m;Z}(t)\overline{V_{\ell',m'}}(s)]=0$ for all $\ell,\ell',m,m',t,s$.

Hence, the sequence $\sum_{\ell=0}^L \sum_{m=-\ell}^\ell \psi_{\ell;j} a_{\ell, m; Z } (t-j) Y_{\ell,m}(x)$ is Cauchy in $L^2(\Omega)$, since for $L>L'$,
$$
\mathbb{E} \left | \sum_{\ell =L'+1}^L \psi_{\ell;j} a_{\ell, m; Z } (t-j) Y_{\ell,m}(x) \right |^2 =  \sum_{\ell =L'+1}^L \frac{2\ell+1}{4\pi} \psi^2_{\ell; j} \sigma^2_\ell \to 0
$$
and $\Psi_j Z(x, t - j)$ is well defined. Similarly $Z(x,t) := \sum_{\ell,m} a_{\ell, m;Z} (t) Y_{\ell,m}(x)$ and $V(x,t)$ are also well defined. 
Moreover, $\sum_{j} \Psi_j Z(x, t - j)$ is Cauchy, indeed for $J > J'$
\begin{align*}
\mathbb{E} \left | \sum_{j=J'+1}^J  \Psi_j Z(x, t - j) \right |^2 &= \lim_{L\to \infty } \mathbb{E} \left | \sum_{j=J'+1}^J \sum_{\ell=0}^L \sum_{m=-\ell}^\ell \psi_{\ell;j} a_{\ell, m; Z } (t-j) Y_{\ell,m}(x) \right |^2 \\ 
&= \sum_{j=J'+1}^J \sum_{\ell=0}^\infty \frac{2\ell+1}{4\pi} \psi^2_{\ell; j} \sigma^2_\ell \to 0
\end{align*}
Then, we have
\begin{align*}
&\mathbb{E} \left | T(x,t) -  \sum_{j=0}^J  \Psi_j Z(x, t - j) - V(x,t) \right |^2 \\=& \lim_{L\to \infty} \mathbb{E}  \left |\sum_{\ell=0}^L \sum_{m=-\ell}^\ell \left ( a_{\ell,m}(t)  - \sum_{j=0}^J \psi_{\ell;j} a_{\ell, m; Z } (t-j)  - V_{\ell,m}(t) \right ) Y_{\ell,m}(x) \right |^2\\
=& \sum_{\ell=0}^\infty \frac{2\ell+1}{4\pi} \sigma^2_\ell \sum_{j=J+1} \psi_{\ell;j}^2,
\end{align*}
and by Dominated Convergence Theorem we conclude the proof.
\end{proof}

\begin{proof}[Proof of Theorem \ref{theo::approx-l2}]
Under Condition \ref{cond::wold} $(ii)$, the Wold decomposition has no deterministic component, that is, $V(x,t) = 0$ for all $(x,t) \in \s \times \mathbb{Z}$.
Then, as a consequence of Lemma \ref{theo::wold}, for all $\epsilon > 0$, there exists integers $L$  and $q$ such that
$$
\mathbb{E}\left | T(x,t) -  Z(x,t) - \sum_{j=1}^q \Psi_{j;L} Z(x, t - j) \right |^2 \le \epsilon.
$$
Moreover, since $a_{\ell,m;Z}(t) \in \mathcal{M}_{\ell,m; t}$ for all $\ell,m,t$, the representation can be inverted:
$$
a_{\ell,m} (t) = \sum_{j=1}^\infty \phi_{\ell;j} a_{\ell,m}(t-j) + a_{\ell,m;Z} (t),
$$
with $\sum_{j=1}^\infty \phi^2_{\ell;j} < \infty$. Then, for all $\epsilon > 0$, there exist integers $L$ and $p$ such that
$$
\mathbb{E}\left | T(x,t) -  \sum_{j=1}^p \Phi_{j;L} T(x, t - j) - Z(x,t) \right |^2 \le \epsilon.
$$
Similarly, it is possible to show the $L^2(\s \times \Omega)$ approximation.
\end{proof}

\bibliographystyle{sapthesis} % BibTeX style
\bibliography{mybiblio} % BibTeX database without .bib extension
 
 \end{document}